\documentclass[12pt]{article}

\usepackage[english]{babel}
\usepackage[cp1250]{inputenc}
\usepackage{graphicx}
\usepackage{url}
\usepackage[all]{xy}
\usepackage{multicol}
\usepackage{amsthm}
\usepackage{amsmath}
\usepackage{amssymb}
\usepackage{geometry}
\usepackage{pinlabel}
\usepackage{hyperref}

\newtheorem{Theorem}{Theorem}[section]
\newtheorem{lemma}[Theorem]{Lemma}
\newtheorem{proposition}[Theorem]{Proposition}
\newtheorem{corollary}[Theorem]{Corollary}
\newtheorem{definition}[Theorem]{Definition}
\newtheorem{remark}[Theorem]{Remark}
\newtheorem{example}[Theorem]{Example}

\newcommand{\ZZ}{\mathbb {Z}}

\newcommand{\tr}{\triangleright}
\newcommand{\zn}{Z_{n}}

\newcommand{\mo}{\rm{ mod}}

\begin{document}
\title{Rack invariants of links in $L(p,1)$}

\author{Eva Horvat \\ email \href{mailto:eva.horvat@pef.uni-lj.si}{eva.horvat@pef.uni-lj.si}}

\date{\today}
\maketitle

\begin{abstract}
We describe a presentation for the augmented fundamental rack of a link in the lens space $L(p,1)$. Using this presentation, the (enhanced) counting rack invariants that have been defined for the classical links are applied to the links in $L(p,1)$. In this case, the counting rack invariants also include the information about the action of $\pi _{1}(L(p,1))$ on the augmented fundamental rack of a link. 
\end{abstract}

\begin{section}{Introduction \& Background}\label{sec1}
While being an interesting algebraic object by themselves, racks and quandles have been extensively used in the study of classical knots and links \cite{MS}. Any link in a 3-manifold has an (augmented) fundamental rack, whose presentation may be obtained from a link diagram $D$, and in \cite{FR} it has been shown that the fundamental rack is a complete invariant for irreducible framed links in a 3-manifold. In \cite{SN1}, Sam Nelson defined computable rack invariants of classical links, which are based on counting the rack homomorphisms from the fundamental rack of a link to a fixed finite rack. 

In this paper we describe the augmented fundamental rack $R(D)$ of a framed link in the lens space $L(p,1)$. We study the set of rack homomorphisms from $R(D)$ to a finite rack $X$, based on which we define the counting rack invariants of links in $L(p,1)$. Using a \verb|Phyton| code, we are able to compute the rack counting invariants of knots and links in $L(p,1)$. The set of rack homomorphisms $Hom(R(D),X)$ also displays the action of $\pi _{1}(L(p,1))$ on the fundamental rack $R(D)$. The information about this action may be used as an additional tool to distinguish links in $L(p,1)$. 

The paper is organized as follows. In subsection \ref{subs11} we review the basic rack terms which will be needed in the rest of the paper. In subsection \ref{subs12} we define the augmented fundamental rack of a link in a 3-manifold, and recall the presentation for the fundamental rack of a classical link. In section \ref{sec2} we describe a general presentation of the augmented fundamental rack $R(D)$ of a link in $L(p,1)$, based on a link diagram $D$. We simplify this presentation and transform it to a finite primary rack presentation. In subsection \ref{subs21}, we study the set of rack homomorphisms from $R(D)$ to a finite rack $X$. By classifying those homomorphisms, we are able to compute the homomorphism set $Hom(R(D),X)$ for a given finite rack $X$. In section \ref{sec3}, we define the counting rack invariants of links in $L(p,1)$. We demonstrate the calculation of these invariants in some examples. In the final section \ref{sec4}, we use the calculation of $Hom(R(D),X)$ to study the action of $\pi _{1}(L(p,1))$ on the augmented fundamental rack $R(D)$. 

\begin{subsection}{The definition of a rack}\label{subs11}
\begin{definition}\label{def1}
A \textbf{rack} is an algebraic structure comprising a nonempty set $R$ with a binary operation $\tr \colon R\times R\to R$ satisfying two axioms:
\begin{enumerate}
\item For any $x,y\in R$ there is a unique $z\in R$ such that $x=z\tr y$.
\item For any $x,y,z\in R$ the formula $(x\tr y)\tr z=(x\tr z)\tr (y\tr z)$ holds.  
\end{enumerate} 
\end{definition}
It follows from the first rack axiom that the map $f_{y}\colon R\to R$ defined by $f_{y}(t)=t\tr y$ is a bijection. We denote its inverse by $f_{y}^{-1}(t)=t\tr \overline{y}$. 

We say that two elements $x,y\in R$ are \textbf{operator equivalent} (which we denote by $x\equiv y$) if the equality $z\tr x=z\tr y$ holds for every $z\in R$. 

\begin{lemma} \label{lemma1} If $y$ is an element of a rack $R$, then $y\tr y\equiv y$. 
\end{lemma}
\begin{proof} Let $x\in R$ be any element of the rack. By the first rack axiom, there is a unique $z\in R$ such that $x=z\tr y$. Applying the second rack axiom we compute $x\tr (y\tr y)=(z\tr y)\tr (y\tr y)=(z\tr y)\tr y=x\tr y$. 
\end{proof}

If the underlying set of a rack is finite, then the rack operation may be encoded in a square matrix with integer entries as follows \cite{SN1}. Let $R$ be a rack with $n$ elements $y_{1},\ldots ,y_{n}$. The \textbf{rack matrix} of $R$ is the matrix $M_{R}\in \ZZ ^{n\times n}$ defined by $M_{R}^{(i,j)}=k$ where $y_{i}\tr y_{j}=y_{k}$ for $i,j=1,2,\ldots n$. Observe that reordering the elements of $R$ changes the rack matrix, so $M_{R}$ is only defined up to a simultaneous permutation of its rows and columns.   Denote by $\zn $ the set $\{k\in \mathbb{Z}|\, 1\leq k\leq n\}$. The following lemma may be compared to \cite[Lemma 1]{SN2}. 

\begin{lemma}An $(n\times n)$ integer matrix $M$ with entries from the set $\zn $ is a rack matrix if and only if the following conditions are satisfied:
\begin{itemize}
\item[a) ]Each column of $M$ is a permutation of the set $\zn $. 
\item[b) ]For any indices $i,j,k\in \zn $ we have $M^{(M^{(i,j)},k)}=M^{(M^{(i,k)},M^{(j,k)})}$. 
\end{itemize}
\end{lemma}
\begin{proof} Let $M$ be an $(n\times n)$ rack matrix with the underlying set $\{y_{1},\ldots ,y_{n}\}$. Then the first rack axiom implies the condition a), and the second rack axiom implies the condition b). 

Now suppose we are given the matrix $M$ with entries from $\zn $, satisfying the conditions a) and b). Define a binary operation $\tr $ on the set $\zn $ by $i\tr j=M^{(i,j)}$. If the conditions a) and b) are satisfied, it follows that $(\zn ,\tr )$ is a rack with the rack matrix $M$. 
\end{proof}
 
We recall the following definition from \cite{SN1}. 
\begin{definition} \label{def2}Let $R$ be a rack. For any $x\in R$ and $n\in \mathbb{N}$ define $x^{\tr n}$ recursively by $x^{\tr 1}=x\tr x$ and $x^{\tr (k+1)}=x^{\tr k}\tr x^{\tr k}$. The \textbf{rack rank} of $x$ is the minimal number $N(x)\in \mathbb{N}$ such that $x^{\tr N(x)}=x$, or $N(x)=\infty $ if $x^{\tr N}\neq x$ for all $N\in \mathbb{N}$. The \textbf{rack rank} of $R$ is then $$N(R)=\textrm{lcm}\{N(x)|\, x\in R\}\;.$$ A rack whose rack rank equals 1 is called a \textbf{quandle}. 
\end{definition} If $R$ is a finite rack, then the rack rank $N(R)$ may be obtained from the rack matrix of $R$, as remarked in \cite{SN1}: the diagonal of the rack matrix is a permutation $\pi \colon R\to R$, given by $\pi (x)=x\tr x$, and the rack rank of $R$ equals the order of $\pi \in S_{|R|}$. 
\end{subsection}

\begin{example} The matrix $\begin{bmatrix}
1 & 1 & 1\\
2 & 3 & 3\\
3 & 2 & 2
\end{bmatrix}$ is a rack matrix of a rack $R$ with $N(R)=2$. 
\end{example}

The concept of rack may be generalized by making the operator group explicit. We recall the following definition from \cite[page 355]{FR}. 
\begin{definition} \label{def3} Let $G$ be a group, acting on itself by conjugation as $g^{h}:=h^{-1}gh$ for any $g,h\in G$. An \textbf{augmented rack} $(X,G)$ is a set $X$ with an action by the group $G$, written as $(x,g)\mapsto x^{g}$ and a function $\partial \colon X\to G$, satisfying the augmentation identity: $$\partial (x^{g})=g^{-1}(\partial x)g\textrm{  for all $x\in X, g\in G$}\;.$$ The rack operation of $X$ on itself is defined by $x\tr y:=x^{\partial (y)}$.
\end{definition}

\begin{definition} \label{def4} Let $(X,G)$ be an augmented rack and let $Y$ be a rack. A function $f\colon X\to Y$ is a \textbf{rack homomorphism} if the following two conditions are satisfied for any $x_{1},x_{2}\in X$:
\begin{enumerate}
\item $f(x_{1}\tr x_{2})=f(x_{1})\tr f(x_{2})$ 
\item $f(x_{1})=f(x_{2})\Leftrightarrow f(x_{1}^{g})=f(x_{2}^{g})$ for any $g\in G$.
\end{enumerate}
\end{definition}
Thus, a rack homomorphism $f$ of an augmented rack $(X,G)$ induces an action on the image $f(X)$ by $(f(x),g)\mapsto f(x^{g})$, which is well defined by the second condition above.  
 
\begin{definition} \label{def5}Let $R$ be a rack. An equivalence relation $\sim $ on $R$ is called a \textbf{congruence} if the condition $(x\sim y,\, z\sim w\, \Rightarrow x\tr z\sim y\tr w)$ is valid for all $x,y,z,w\in R$. 
\end{definition} 
We introduce the general rack presentations, following the definition in \cite[page 376]{FR}. 
\begin{definition} \label{def6} For two sets $A$ and $B$, denote by $F(A\cup B)$ the free group on $A\cup B$. The \textbf{extended free rack} $FR(A,B)$ is the augmented rack $(A\times F(A\cup B), F(A\cup B))$, where the map $\partial \colon A\times F(A\cup B)\to F(A\cup B)$ is given by $\partial (a^{w}):=w^{-1}aw$. The augmentation identity implies that the rack operation is given by $$(a^{w})\tr (c^{z})=(a^{w})^{\partial (c^{z})}=a^{wz^{-1}cz}$$ for any $a,c\in A$ and $w,z\in F(A\cup B)$.

The \textbf{free rack} on the set $A$ is then defined as $FR(A):=FR(A,\{\})$. 
\end{definition} 
A general rack presentation is given by the following sets:
\begin{itemize}
\item the set of the primary generators $S_{P}$,
\item the set of the operator generators $S_{O}$, 
\item the set of the primary relations $R_{P}$ and 
\item the set of the operator relations $R_{O}$.
\end{itemize} 
\begin{definition}\label{def7} The rack generated by the presentation $[S_{P},S_{O}:R_{P},R_{O}]$ is the quotient $\frac{FR(S_{P},S_{O})}{\sim }$, where $\sim $ is the smallest congruence on $FR(S_{P},S_{O})$ containing 
\begin{itemize}
\item $x\sim y$ if $(x=y)\in R_{P}$
\item $z\tr x\sim z\tr y$ for each $z\in FR(S_{P},S_{O})$ if $(x=y)\in R_{P}$
\item $z^{u}\sim z^{v}$ for each $z\in FR(S_{P},S_{O})$ if $u\equiv v\in R_{O}$.
\end{itemize}
\end{definition}
A general rack presentation without operator generators ($S_{O}=\emptyset $)  is called a primary rack presentation. 

\begin{subsection}{The fundamental rack of a link in a 3-manifold}\label{subs12}
We briefly summarize the following from \cite[page 358]{FR} for the reader's convenience. Let $Q$ be a closed connected orientable 3-manifold and let $L$ be a link in $Q$. The link is \textbf{framed} if a cross section $\lambda \colon L\to \partial N(L)$ of the normal disk bundle is chosen. We call this cross section a \textbf{framing} and denote its image by $L^{+}=\lambda (L)$.  

Let $Q_{0}=\textrm{closure}(Q-N(L))$ and choose a basepoint $x_{0}\in Q_{0}$. Denote $$\Gamma (L)=\{\textrm{homotopy classes of paths in $Q_{0}$ from a point in $L^{+}$ to $x_{0}$}\}\;.$$ During the homotopy the initial point may move around on $L^{+}$, while the final point is kept fixed. 

The fundamental group $\pi _{1}(Q_{0},x_{0})$ acts on the set $\Gamma (L)$: let $a\in \Gamma (L)$ be represented by a path $\alpha $ and $g\in \pi _{1}(Q_{0},x_{0})$ be represented by a loop $\gamma $, then $a^{g}:=[\alpha \cdot \gamma ]$, where $\cdot $ denotes concatenation of paths. 

Using this action, the set $\Gamma (L)$ may be equipped with a structure of an augmented rack. Any $p\in L^{+}$ lies on a unique meridian circle of the normal circle bundle and we denote by $m_{p}$ the loop based at $p$ which follows around the meridian in the positive direction. 

\begin{definition}\label{def8}The \textbf{augmented fundamental rack} of the framed link $L$ is the augmented rack $(\Gamma (L),\pi _{1}(Q_{0},x_{0}))$ as above with the function $\partial \colon \Gamma (L)\to \pi _{1}(Q_{0},x_{0})$ defined as follows. Given two classes $a,b\in \Gamma $, which are represented by the paths $\alpha $ and $\beta $ respectively, define $\partial (b)=\overline{\beta }\cdot m_{\beta (0)}\cdot \beta$, where $\cdot $ denotes the concatenation of paths. This produces the rack operation $$a\tr b=a ^{\partial (b)}=[\alpha \cdot \overline{\beta }\cdot m_{\beta (0)}\cdot \beta ]\;.$$  
\end{definition}

\begin{example}\label{ex1}\textbf{The fundamental rack of a classical link.} 
Let $L$ be a framed link in $S^{3}$, given by a blackboard-framed link diagram $D$. Label the arcs of $D$ by $y_{1},\ldots ,y_{n}$. For each arc $y_{i}$ of the diagram $D$, choose a path $x_{i}$ from a point on $\lambda (y_{i})$ to the basepoint, where $\lambda\colon L\to \partial N(L)$ is the chosen framing. Then the homotopy classes of the paths $x_{1},\ldots ,x_{n}$ are the generators of the fundamental rack of $L$. 
\begin{figure}[h]
\labellist
\normalsize \hair 2pt
\pinlabel $x_{k}$ at 260 0
\pinlabel $\partial (x_{j})$ at 410 100
\pinlabel $x_{i}$ at 450 220
\pinlabel $y_{k}$ at 80 120
\pinlabel $y_{j}$ at 170 320
\pinlabel $y_{i}$ at 350 380
\pinlabel $*$ at 530 10
\endlabellist
\begin{center}
\includegraphics[scale=0.3]{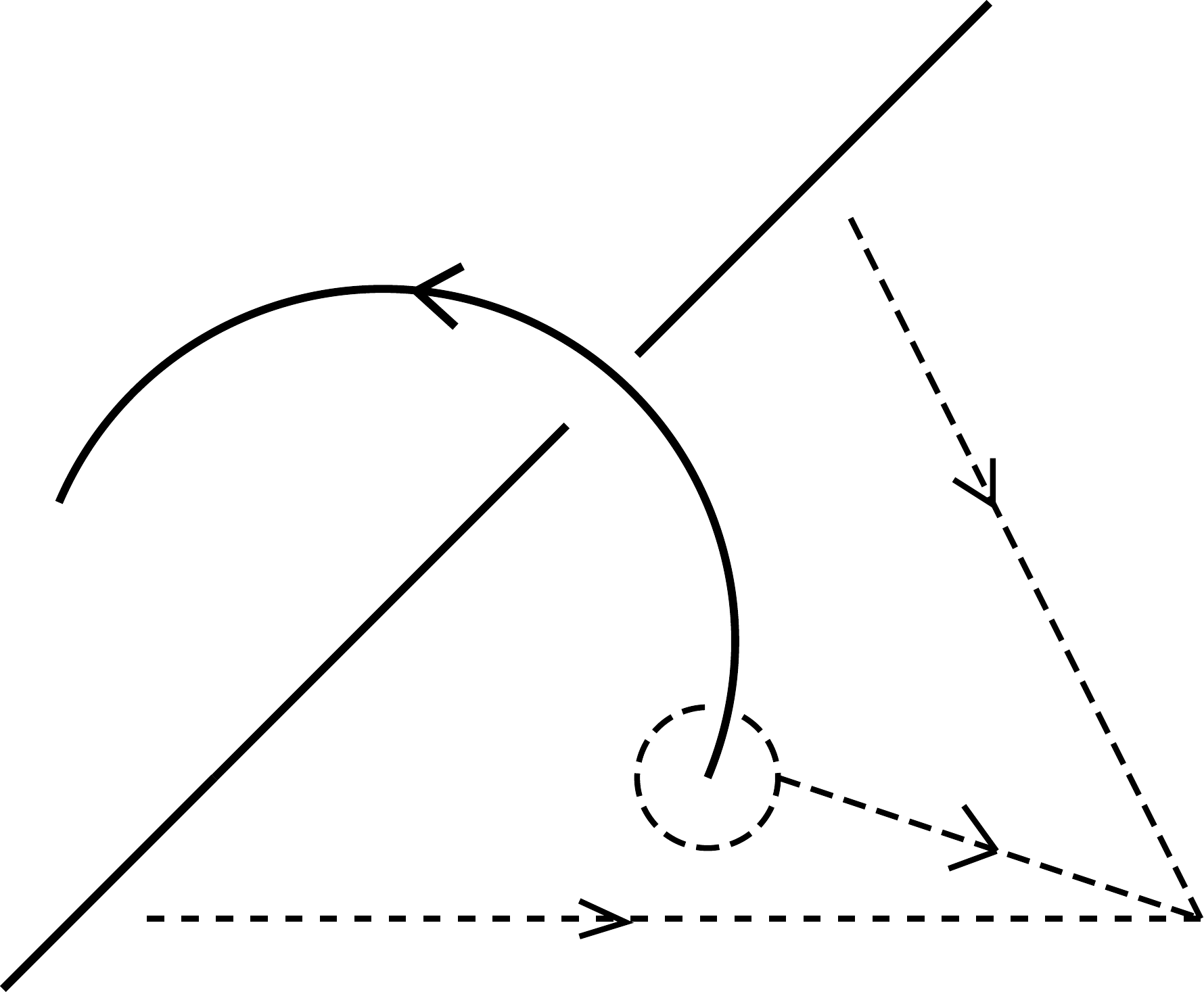}
\caption{The illustration of the crossing relation}
\label{fig:relation}
\end{center}
\end{figure}
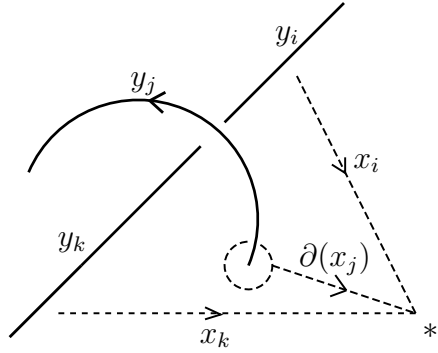
The augmentation map $\partial \colon \Gamma (L)\to \pi _{1}(S^{3}\backslash N(L),*)$ is given by $\partial (x_{i})=\overline{x}_{i}\cdot m_{x_{i}(0)}\cdot x_{i}$ and the fundamental group $\pi _{1}(S^{3}\backslash N(L),*)$ is actually generated by the images $\partial (x_{1}),\ldots ,\partial (x_{n})$. 
Each crossing of the diagram $D$ consists of an overcrossing arc $y_{j}$ and two arcs $y_{i}$ and $y_{k}$ of the undercrossing strand. If we see the arc $y_{i}$ on the right (and $y_{k}$ on the left) of $y_{j}$ when passing along $y_{j}$, then we may see the homotopy between $x_{i}^{\partial (x_{j})}$ and $x_{k}$ in the Figure \ref{fig:relation}. This homotopy implies the crossing relation $x_{i}\tr x_{j}=x_{k}$ in the fundamental rack of $L$. The fundamental rack of $L$ then has a primary presentation $$\left [x_{1},\ldots ,x_{n}:\textrm{crossing relations of $D$}\right ]\;.$$ 
\end{example}

\end{subsection}
\end{section}

\begin{section}{The fundamental augmented rack of a link in $L(p,1)$}\label{sec2}

Let $L$ be a framed link in the lens space $L(p,1)$. The lens space is obtained by integral surgery on an unknot $U$ in $S^{3}$. We specify the framing coefficient $p$ of the surgery along $U$ by twisting the unknot $p$ times in the positive direction. Thus there are $p$ positive self-crossings of $U$ in the diagram and the framing is given as the blackboard framing. We draw the framed link $L$  inside the diagram containing the surgery curve $U\subset S^{3}$ and denote the resulting diagram by $D$. Let $m$ be the number of crossings where both the underarc and the overarc belong to the link $L$. Label the arcs of the link $L$ in the diagram by $x_{1},\ldots ,x_{m+d}$ and label the arcs of the surgery curve $U$ by $a,a_{1},\ldots ,a_{d-1}$ as shown in the Figure \ref{fig:diagram}.  
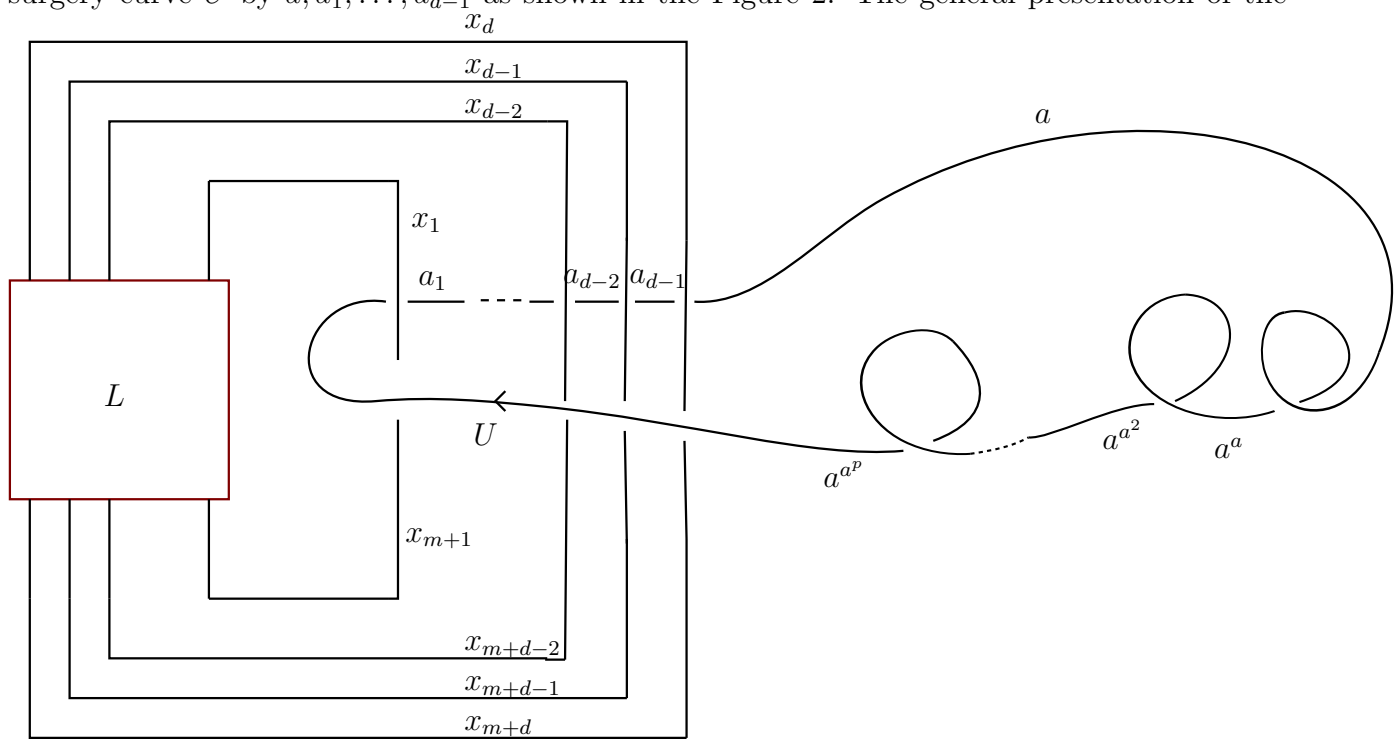
\begin{figure}[h]
\labellist
\normalsize \hair 2pt
\pinlabel $a^{a^{2}}$ at 840 230
\pinlabel $a^{a}$ at 920 220
\pinlabel $a^{a^{p}}$ at 630 200
\pinlabel $a$ at 780 470
\pinlabel $a_{1}$ [b] at 320 335
\pinlabel $a_{d-2}$ [b] at 440 335
\pinlabel $a_{d-1}$ [b] at 490 335
\pinlabel $x_{d}$ [l] at 340 540
\pinlabel $x_{1}$ [l] at 300 390
\pinlabel $x_{d-2}$ [l] at 340 475
\pinlabel $x_{d-1}$ [l] at 340 505
\pinlabel $x_{m+1}$ [b] at 325 140
\pinlabel $x_{m+d-2}$ [l] at 340 70
\pinlabel $x_{m+d-1}$ [l] at 340 40
\pinlabel $x_{m+d}$ [l] at 340 10
\pinlabel $L$ at 80 260
\pinlabel $U$ at 360 230
\endlabellist
\begin{center}
\includegraphics[scale=0.50]{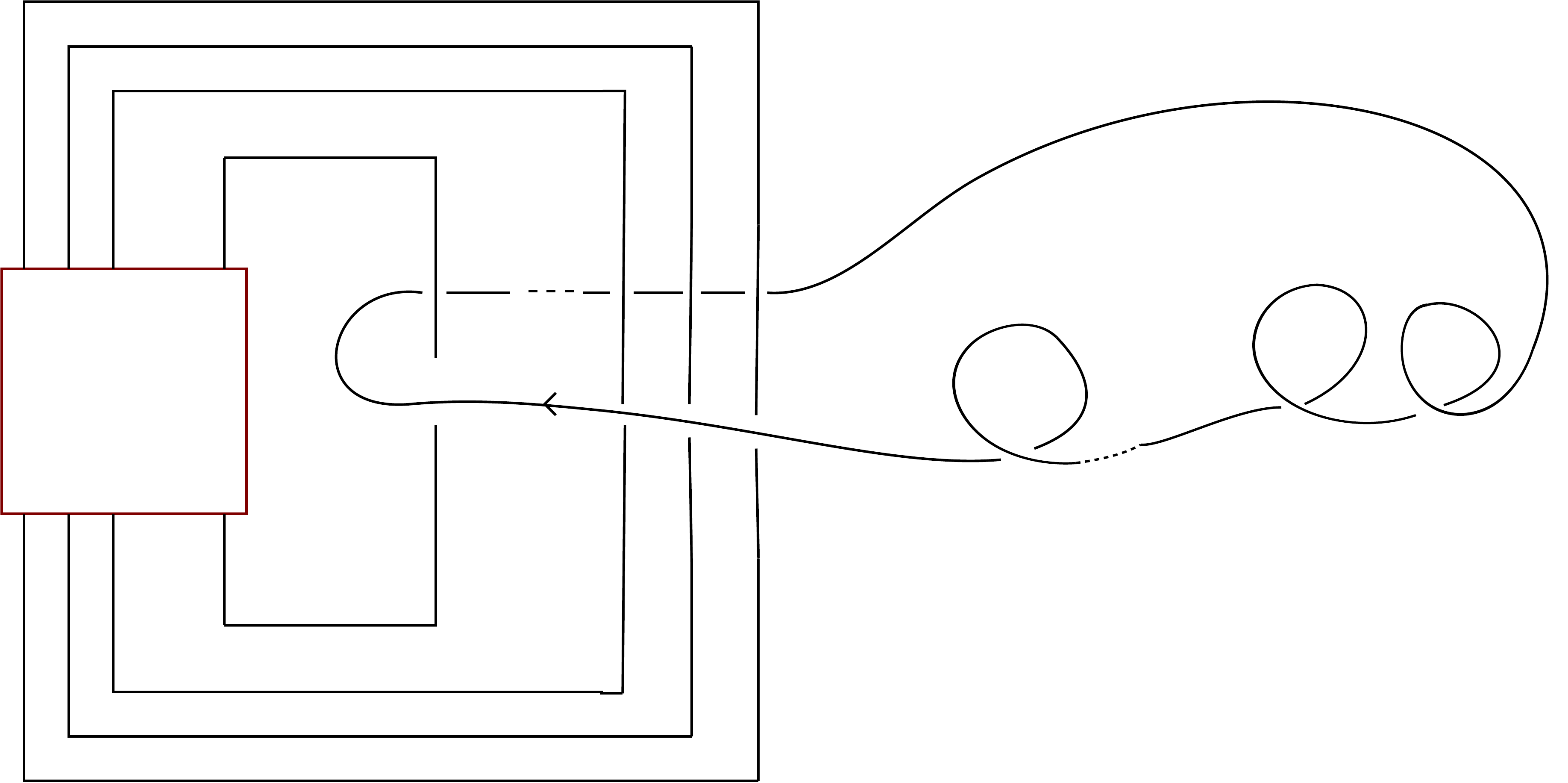}
\caption{The diagram $D$ of a framed link in $L(p,1)$}
\label{fig:diagram}
\end{center}
\end{figure}
The general presentation of the augmented fundamental rack of $L$ is obtained from the diagram $D$ as follows \cite[page 379]{FR}: \begin{itemize}
\item[(i)] The arcs of the link $L$ represent the primary generators $x_{1},\ldots ,x_{m+d}$
\item[(ii)] The arcs of the surgery curve $U$ represent the operator generators $a,a_{1},\ldots ,a_{d-1}$
\item[(iii)] At each crossing where the underarc belongs to the link $L$, the crossing relation gives a primary relator  
\item[(iv)] At each crossing where the underarc belongs to the surgery curve $U$, the crossing relation gives an operator relator. Finally we obtain another operator relator by reading around the surgery curve and noting the (signed) undercrossings. 
\end{itemize} We obtain the following general presentation: 
\begin{xalignat*}{1}
& [\{x_{1},\ldots ,x_{m+d}\},\{a,a_{1},\ldots ,a_{d-1}\}:\{\textrm{crossing relations of $L$},x_{i}^{a^{(a^{p})}}=x_{m+i}\textrm{ for $i=1,\ldots ,d$}\},\\
& \{\overline{x}_{1}^{\epsilon _{1}}a_{1}x_{1}^{\epsilon _{1}}\equiv a^{(a^{p})},\overline{x}_{2}^{\epsilon _{2}}a_{2}x_{2}^{\epsilon _{2}}\equiv a_{1},\ldots ,\overline{x}_{d}^{\epsilon _{d}}ax_{d}^{\epsilon _{d}}\equiv a_{d-1},\, aa^{a}\ldots a^{(a^{p-1})}\, \overline{x}_{1}^{\epsilon _{1}}\overline{x}_{2}^{\epsilon _{2}}\ldots \overline{x}_{d}^{\epsilon _{d}}\equiv 1\}]\;.
\end{xalignat*} 
where $\epsilon _{i}=1$ if the arc $x_{i}$ follows the arc $x_{m+i}$ in the chosen orientation of $L$, and $\epsilon _{i}=-1$ otherwise. 
Using the operator relations, the operator generators $a_{1},\ldots ,a_{d-1}$ are expressed in terms of the operator generator $a$ and thus may be omitted from the presentation. The remaining relation is calculated as $$a^{(a^{p})}\equiv \overline{x}_{1}^{\epsilon _{1}}a_{1}x_{1}^{\epsilon _{1}}\equiv \overline{x}_{1}^{\epsilon _{1}}\overline{x}_{2}^{\epsilon _{2}}a_{2}x_{2}^{\epsilon _{2}}x_{1}^{\epsilon _{1}}\equiv \ldots \equiv \overline{x}_{1}^{\epsilon _{1}}\overline{x}_{2}^{\epsilon _{2}}\ldots \overline{x}_{d}^{\epsilon _{d}}ax_{d}^{\epsilon _{d}}\ldots x_{2}^{\epsilon _{2}}x_{1}^{\epsilon _{1}}\;.$$ 
By the definition of an augmented rack \ref{def3}, the action of the operator group on itself is taken to be conjugation. It follows that $a^{(a^{k})}\equiv (a^{k})^{-1}a(a^{k})\equiv a$ for any $k\in \mathbb{N}$. The presentation thus simplifies to
\begin{xalignat}{1} \label{pres1}
& [\{x_{1},\ldots ,x_{m+d}\},\{a\}:\{\textrm{crossing relations of $L$}, x_{i}^{a}=x_{m+i} \textrm{ for $i=1,\ldots ,d$}\},\\
& \{a^{p}\equiv x_{d}^{\epsilon _{d}}\ldots x_{2}^{\epsilon _{2}}x_{1}^{\epsilon _{1}}\}] \nonumber
\end{xalignat}
There is a class of links in $L(p,1)$ which we need to consider separately: 
\begin{definition} \label{def9}A link $L\subset L(p,1)$ is \textbf{affine} (sometimes also called local) if it is contained in a 3-ball $B^{3}\subset L(p,1)$.  
\end{definition}
If $L$ is an affine link, then its diagram $D$ may be drawn without intersections between $L$ and the surgery curve $U$. Thus $d=0$ and the presentation of $R(D)$ is very simple: appart from the crossing relations of $L$, there is only one operator relation $a^{p}\equiv 1$. 

We summarize our results in the following proposition.  
\begin{proposition} \label{prop1} Let $L$ be a framed link in $L(p,1)$, given by the diagram in the Figure \ref{fig:diagram}. Then the augmented fundamental rack of $L$ is given by a presentation 
\begin{xalignat*}{1}
& [\{x_{1},\ldots ,x_{m+d}\},\{a\}:\{\textrm{crossing relations of $L$}, x_{i}^{a}=x_{m+i} \textrm{ for $i=1,\ldots ,d$}\},\{a^{p}\equiv x_{d}^{\epsilon _{d}}\ldots x_{2}^{\epsilon _{2}}x_{1}^{\epsilon _{1}}\}] 
\end{xalignat*} if $d>0$, and 
\begin{xalignat*}{1} 
& [\{x_{1},\ldots ,x_{m+d}\},\{a\}:\{\textrm{crossing relations of $L$}\},\{a^{p}\equiv 1\}] 
\end{xalignat*} otherwise. 
\end{proposition}

\begin{remark} Erasing the surgery curve $U$ in our diagram, we obtain an ordinary diagram of a classical link in $S^{3}$. The presentation of the fundamental rack of this classical link may be given as  
\begin{xalignat*}{1}
& [\{x_{1},\ldots ,x_{m+d}\}:\, \{\textrm{crossing relations of $L$},\, x_{i}=x_{m+i}\textrm{ for $i=1,\ldots ,d$}\}]\;,
\end{xalignat*} to see our presentation as a generalization of the classical case. By removing the unnecessary generators $x_{m+1},\ldots ,x_{m+d}$ we obtain the presentation of Example \ref{ex1}. 
\end{remark}

If a framed link in $L(p,1)$ is given by the diagram $D$, then its fundamental rack $R(D)$ has the general presentation, given in the proposition \ref{prop1}. The operator generator $a$ is intrinsic to the ambient space $L(p,1)$ and is not itself an element of $R(D)$; rather it defines an action $A\colon R(D)\to R(D)$ by $A(x)=x^{a}$. 
Define a mapping $F\colon R(D)\to R(D)$ by  
\[
 F(x) = 
  \begin{cases} 
   x\tr x_{d}^{\epsilon _{d}}\tr \ldots \tr x_{1}^{\epsilon _{1}} & \text{if } d > 0\;, \\
   x       & \text{if } d = 0\;.
  \end{cases}
\]

\begin{lemma}\label{lemma2}Let $R(D)$ be given by the general presentation in the proposition \ref{prop1}. Then $A$ and $F$ are rack automorphisms of $R(D)$ for which $A^{p}=F$. 
\end{lemma}
\begin{proof} 
By the definition of the rack operation in the extended free rack \ref{def6}, we have $$A(x)\tr A(y)=(x^{a})\tr (y^{a})=x^{aa^{-1}ya}=x^{ya}=A(x\tr y)\;.$$ By the second rack axiom we have $F(x\tr y)=F(x)\tr F(y)$. It follows from the first rack axiom that $F$ is a bijection, whose inverse for $d>0$ is given by $F^{-1}(x)=x\tr (\overline{x}_{1}^{\epsilon _{1}}\ldots \overline{x}_{d}^{\epsilon _{d}})$, while $A$ is a bijection with the inverse $A^{-1}(x)=x^{\overline{a}}$. 

If $d=0$, then $F=A^{p}=id_{R(D)}$ by the proposition \ref{prop1}, thus the remaining statements of the lemma are obvious.  

Now suppose that $d>0$. The operator relation $a^{p}\equiv x_{d}^{\epsilon _{d}}\ldots x_{1}^{\epsilon _{1}}$ implies that $A^{p}(x)=x^{a^{p}}=F(x)$ for any $x\in R(D)$. 
\end{proof}

The operator relations of the presentation \eqref{pres1} imply that the fundamental augmented rack $R(D)$ actually has a finite primary presentation. 
\begin{proposition} \label{prop2} Let $S=\{(x_{i},k)|\, 1\leq i\leq m+d,\, 0\leq k\leq p-1\}=\{x_{1},x_{2},\ldots ,x_{m+d}\}\times \ZZ _{p}$. The fundamental augmented rack $R(D)$ has a primary presentation 
\begin{xalignat}{1}\label{pres2}
& [S:\, (x_{i},k)\tr (x_{j},k)=(x_{l},k)\textrm{ for every crossing relation $x_{i}\tr x_{j}=x_{l}$ and $k=0,\ldots ,p-1$},\\
& \, (x_{m+i},k)=(x_{i},k+1), \textrm{ for $i=1,\ldots ,d$ and $k=0,\ldots ,p-2$},\, \nonumber \\
& (x_{m+i},p-1)=(x_{i},0)\tr (x_{d}^{\epsilon _{d}},0)\tr \ldots \tr (x_{1}^{\epsilon _{1}},0)\textrm{ for $i=1,\ldots ,d$}]\nonumber
\end{xalignat} if $d>0$, and 
\begin{xalignat*}{1}
& [S:\, (x_{i},k)\tr (x_{j},k)=(x_{l},k)\textrm{ for every crossing relation $x_{i}\tr x_{j}=x_{l}$ and $k=0,\ldots ,p-1$}]
\end{xalignat*} if $d=0$. 
\end{proposition}
\begin{proof} $R(D)$ has a general presentation, given by the proposition \ref{prop1}. Every element $x\in R(D)$ is of the form $x_{i}^{w}$ for some word $w\in F(\{x_{1},\ldots ,x_{m+d},a\})$. Thus $x$ may be written as a finite product of the elements $x_{j}^{(a^{k})}$ for $1\leq j\leq m+d$ and $0\leq k\leq p-1$, since the second operator relation implies that $x_{j}^{(a^{p})}=x_{j}\tr x_{d}^{\epsilon _{d}}\tr \ldots \tr x_{1}^{\epsilon _{1}}$. Denote by $(x_{i},k)\in \{x_{1},x_{2},\ldots ,x_{m+d}\}\times \ZZ _{p}$ the generator of $R(D)$, corresponding to the element $x_{i}^{(a^{k})}$. 

We have shown in the lemma \ref{lemma2} that $A$ is a rack automomorphism, thus we have $A(x\tr z)=A(x)\tr A(z)$ for any $x,z\in S$. Each primary relation of \eqref{pres1} therefore induces $p$ relations of the primary presentation \eqref{pres2}. Every crossing relation $x_{i}\tr x_{j}=x_{l}$ induces relations $(x_{i},k)\tr (x_{j},k)=(x_{l},k)$ for $k=0,\ldots ,p-1$. For  $i=1\ldots ,d$, the relation $x_{m+i}=x_{i}^{a}$ of the presentation \eqref{pres1} induces relations $(x_{m+i},k)=(x_{i},k+1)$ for $k=0,\ldots ,p-2$ and the final relation $(x_{m+i},p-1)=(x_{i},0)\tr (x_{d}^{\epsilon _{d}},0)\tr \ldots \tr (x_{1}^{\epsilon _{1}},0)$.
\end{proof}

\begin{remark} If the fundamental rack $R(D)$ is given by the presentation \eqref{pres2}, then the action of the rack automorphism $A\colon R(D)\to R(D)$ is given by \[A(x_{i},k) = \left\{
  \begin{array}{lr}
    (x_{i},k+1) & \textrm{ for } 0\leq k<(p-1),\\
   (x_{i},0)\tr (x_{d}^{\epsilon _{d}},0)\tr \ldots \tr (x_{1}^{\epsilon _{1}},0) & \textrm{ for $k=p-1$. }
  \end{array}
\right.
\] 
\end{remark}
 
\begin{subsection}{Labeling the homomorphisms from $R(D)$ to a finite rack}\label{subs21}
Let $X$ be a finite rack. We will study the set of homomorphisms $Hom(R(D),X)$, where $R(D)$ denotes the augmented fundamental rack of a framed link given by the diagram $D$. 

If $L$ is a classical link, then its fundamental rack is a quotient of the free rack $FR(\{x_{1},\ldots ,x_{m}\})$ whose generators correspond to the arcs of the diagram $D$. Thus, any rack homomorphism $f\colon R(D)\to X$ is completely specified by the images $f(x_{i})$, $i=1,\ldots ,m$. This is usually thought of as coloring each arc $x_{i}$ by a color from $X$. A function $f\colon \{x_{1},\ldots ,x_{m}\}\to X$ defines a coloring if and only if it preserves the crossing relation at every crossing of the diagram $D$. 

Now let $L$ be a framed link in $L(p,1)$, whose diagram $D$ is labeled as shown in the Figure \ref{fig:diagram}. By the proposition \ref{prop2}, the fundamental rack $R(D)$ is a quotient of the free rack $$FR\left (\{(x_{i},k)|\, 1\leq i\leq m+d,\, 0\leq k\leq p-1\}\right )\;,$$ which implies that any rack homomorphism $f\colon R(D)\to X$ is completely determined by the images $f((x_{i},k))$ for $i=1,\ldots ,m+d$ and $k=0,1,\ldots ,p-1$. We may think of this as choosing for each arc $x_{i}$ of the diagram $D$ a $p$-tuple of colors from $X$, or, equivalently, as choosing $p$ colorings of the diagram $D$ by the colors from $X$. Denote again the generating set of $R(D)$ by $S=\{x_{1},\ldots ,x_{m+d}\}\times \ZZ _{p}$. A function $f\colon S\to X$ defines a vector function $(f_{0},\ldots ,f_{p-1})\colon \{x_{1},\ldots ,x_{m+d}\}\to X^{p}$ by $f_{k}(x_{i})=f((x_{i},k))$. When will this function define a coloring? 

\begin{lemma}\label{lemma3} If a map $f\colon R(D)\to X$ satisfies the condition $f(x\tr y)=f(x)\tr f(y)$ for any $x,y\in R(D)$, then for any pair of elements $z,w\in R(D)$ the following implication holds: $$f(A^{k}(z))=f(A^{k}(w))\textrm{ for }0\leq k\leq p-1\quad \Rightarrow \quad f(A^{k}(z))=f(A^{k}(w))\textrm{ for every }k\in \mathbb{N}\;.$$ 
\end{lemma}
\begin{proof} Suppose that $f(A^{k}(z))=f(A^{k}(w))$ for every $0\leq k\leq p-1$. For $K\geq p$, we may write $K=p+r$ for some $r<K$ and use the lemma \ref{lemma2} to calculate 
\begin{xalignat*}{1}
& f(A^{K}(z))=f(A^{r}(F(z)))=f(F(A^{r}(z)))=f(A^{r}(z))\tr f(x_{d}^{\epsilon _{d}})\tr \ldots \tr f(x_{1}^{\epsilon _{1}})\\
& f(A^{K}(w))=f(A^{r}(F(w)))=f(F(A^{r}(w)))=f(A^{r}(w))\tr f(x_{d}^{\epsilon _{d}})\tr \ldots \tr f(x_{1}^{\epsilon _{1}})
\end{xalignat*} This completes the proof by induction.
\end{proof}

\begin{lemma}\label{lemma4} If $f\colon R(D)\to X$ is a rack homomorphism, then we have  \[f(A^{k}(x_{i},j)) = \left\{
  \begin{array}{lr}
    f_{j+k}(x_{i}) & \textrm{ if } j+k\leq (p-1),\\
    f_{(j+k)\mo \, p}(x_{i})\tr f_{(j+k)\mo \, p}(x_{d}^{\epsilon _{d}})\tr \ldots \tr f_{(j+k)\mo \,p}(x_{1}^{\epsilon _{1}}) & \textrm{ otherwise. }
  \end{array}
\right.
\] for any $1\leq i\leq m+d$ and $0\leq j,k\leq p-1$.
\end{lemma}
\begin{proof} The equality follows from the lemma \ref{lemma2} and the definition of $f$.  
\end{proof}

\begin{proposition} \label{prop3} Let $X$ be a finite rack. Let $L$ be a framed affine link in $L(p,1)$, whose diagram $D$ is labeled as shown in the Figure \ref{fig:diagram} (for $d=0$). A function $f\colon S\to X$ defines a rack homomorphism from $R(D)$ to $X$ if and only if the following conditions are satisfied:
\begin{itemize}
\item[(i)]$f_{k}(x_{i})\tr f_{k}(x_{j})=f_{k}(x_{l})$ for every crossing relation $x_{i}\tr x_{j}=x_{l}$ of the link $L$ and for $k=0,\ldots ,p-1$. 
\item[(ii)]$f(x)=f(y)\Leftrightarrow f(A^{k}(x))=f(A^{k}(y))$ for any $x,y\in S$ and for any $0\leq k\leq p-1$.
\end{itemize}
\end{proposition}
\begin{proof} By the proposition \ref{prop2}, the augmented fundamental rack $R(D)$ has the generating set $S=\{(x_{i},k)|\, 1\leq i\leq m+d,\, 0\leq k\leq p-1\}$ and a presentation 
\begin{xalignat*}{1}
& [S:\, (x_{i},k)\tr (x_{j},k)=(x_{l},k)\textrm{ for every crossing relation $x_{i}\tr x_{j}=x_{l}$ and $k=0,\ldots ,p-1$}]\;.
\end{xalignat*} 

Let $f\colon S\to X$ be a function that satisfies the conditions $(i)$ and $(ii)$. The function $f$ defines a rack homomorphism from the free rack $FR(S)$ to $X$ by $$f\left ((x_{i},r)\tr (x_{j},s)\right ):=f(x_{i},r)\tr f(x_{j},s)$$ for any $(x_{i},r),(x_{j},s)\in S$. Since $f$ satisfies the condition $(i)$, it follows that $f(x_{i},k)\tr f(x_{j},k)=f_{k}(x_{i})\tr f_{k}(x_{j})=f_{k}(x_{l})=f(x_{l},k)$ for every crossing relation $x_{i}\tr x_{j}=x_{l}$ and $k=0,\ldots ,p-1$. Therefore $f$ preserves all the relations of the above presentation. By definition \ref{def4}, it remains to show that $f$ preserves the action of the rack automorphism $A$. Let $f(x_{i},r)=f(x_{j},s)$ for some $(x_{i},r),(x_{j},s)\in S$. Then it follows from $(ii)$ that $f(A^{k}(x_{i},r))=f(A^{k}(x_{j},s))$ for $0\leq k\leq p-1$. Observe that by lemma \ref{lemma2} we have $A^{p}=F=Id_{R(D)}$ and thus
\begin{xalignat*}{1}
& f(A^{k+p}(x_{i},r))=f(A^{k}(F(x_{i},r)))=f(A^{k}(x_{i},r))=f(A^{k}(x_{j},s))=f(A^{k+p}(x_{j},s))
\end{xalignat*} for $k=0,\ldots ,p-1$. It follows by induction that $f(A^{n}(x_{i},r))=f(A^{n}(x_{j},s))$ for any $n\in \ZZ $, therefore $f$ preserves the action of the automorphism $A$ and thus defines a rack homomorphism from $R(D)$ to $X$.  

Conversely, let $f\colon S\to X$ be a function that defines a rack homomorphism from $R(D)$ to $X$. Then $f$ preserves all the relations of the above presentation, which implies $(i)$, and  $f$ preserves the action of $A$, which implies $(ii)$. 
\end{proof}

\begin{proposition} \label{prop4} Let $X$ be a finite rack. Let $L$ be a framed link in $L(p,1)$, which is not affine and whose diagram $D$ is labeled as shown in the Figure \ref{fig:diagram}.  A function $f\colon S\to X$ defines a rack homomorphism from $R(D)$ to $X$ if and only if the following conditions are satisfied:
\begin{itemize}
\item[(i)]$f_{k}(x_{i})\tr f_{k}(x_{j})=f_{k}(x_{l})$ for every crossing relation $x_{i}\tr x_{j}=x_{l}$ of the link $L$ and for $k=0,\ldots ,p-1$. 
\item[(ii)]$f_{0}(x_{m+i})=f_{1}(x_{i})$ for $i=1,\ldots ,d$..  
\item[(iii)]$f(x)=f(y)\Leftrightarrow f(A^{k}(x))=f(A^{k}(y))$ for any $x,y\in S$ and for any $0\leq k\leq p-1$.
\end{itemize}
\end{proposition}
\begin{proof} By the proposition \ref{prop2}, the augmented fundamental rack $R(D)$ has a presentation 
\begin{xalignat*}{1}
& [S:\, (x_{i},k)\tr (x_{j},k)=(x_{l},k)\textrm{ for every crossing relation $x_{i}\tr x_{j}=x_{l}$ and $k=0,\ldots ,p-1$},\\
& \, (x_{m+i},k)=(x_{i},k+1), \textrm{ for $i=1,\ldots ,d$ and $k=0,\ldots ,p-2$}, \\
& (x_{m+i},p-1)=(x_{i},0)\tr (x_{d}^{\epsilon _{d}},0)\tr \ldots \tr (x_{1}^{\epsilon _{1}},0)\textrm{ for $i=1,\ldots ,d$}]
\end{xalignat*}

Let $f\colon S\to X$ be a function that satisfies the conditions $(i)$-$(iii)$. The function $f$ defines a rack homomorphism from the free rack $FR(S)$ to $X$ by $$f\left ((x_{i},r)\tr (x_{j},s)\right ):=f(x_{i},r)\tr f(x_{j},s)$$ for any $(x_{i},r),(x_{j},s)\in S$. Since $f$ satisfies the condition $(i)$, it follows that $f(x_{i},k)\tr f(x_{j},k)=f_{k}(x_{i})\tr f_{k}(x_{j})=f_{k}(x_{l})=f(x_{l},k)$ for every crossing relation $x_{i}\tr x_{j}=x_{l}$ and $k=0,\ldots ,p-1$. Since $f$ satisfies the condition $(ii)$, we have $f(x_{m+i},0)=f_{0}(x_{m+i})=f_{1}(x_{i})=f(x_{i},1)$, and by the condition $(iii)$ it follows that $f(x_{m+i},k)=f(A^{k}(x_{m+i},0))=f(A^{k}(x_{i},1))=f(x_{i},k+1)$ for $k=0,\ldots ,p-2$ and 
\begin{xalignat*}{1}
& f(x_{m+i},p-1)=f(A^{p-1}(x_{m+i},0))=f(A^{p-1}(x_{i},1))=f((x_{i},0)\tr (x_{d}^{\epsilon _{d}},0)\tr \ldots \tr (x_{1}^{\epsilon _{1}},0))=\\
& =f(x_{i},0)\tr f(x_{d}^{\epsilon _{d}},0)\tr \ldots \tr f(x_{1}^{\epsilon _{1}},0)
\end{xalignat*} for $i=1,\ldots ,d$. Therefore $f$ preserves all the relations of the above presentation. To see that $f$ preserves the action of the rack automorphism $A$, let $f(x_{i},r)=f(x_{j},s)$ for some $(x_{i},r),(x_{j},s)\in S$. It follows from $(iii)$ that $f(A^{k}(x_{i},r))=f(A^{k}(x_{j},s))$ for $0\leq k\leq p-1$. By lemma \ref{lemma2} we have
\begin{xalignat*}{1}
& f(A^{k+p}(x_{i},r))=f(F(A^{k}(x_{i},r)))=f(A^{k}(x_{i},r)\tr (x_{d}^{\epsilon _{d}},0)\tr \ldots \tr (x_{1}^{\epsilon _{1}},0))=\\
& =f(A^{k}(x_{i},r))\tr f(x_{d}^{\epsilon _{d}},0)\tr \ldots \tr f(x_{1}^{\epsilon _{1}},0)=f(A^{k}(x_{j},s))\tr f(x_{d}^{\epsilon _{d}},0)\tr \ldots \tr f(x_{1}^{\epsilon _{1}},0)=\\
& =f(A^{k+p}(x_{j},s))
\end{xalignat*} for $k=0,\ldots ,p-1$. It follows by induction that $f(A^{n}(x_{i},r))=f(A^{n}(x_{j},s))$ for any $n\in \ZZ $, therefore $f$ preserves the action of $A$ and defines a rack homomorphism from $R(D)$ to $X$. 

Conversely, let $f\colon S\to X$ be a function that defines a rack homomorphism from $R(D)$ to $X$. Then $f$ preserves all the relations of the above presentation, which implies $(i)$ and $(ii)$, and  $f$ preserves the action of $A$, which implies $(iii)$.
\end{proof}

Let $L$ be a link in $L(p,1)$, given by the diagram on the Figure \ref{fig:diagram}. The surgery curve $U$ bounds an obvious disk $\mathcal{D}$ in $S^{3}$. Cutting the link $L$ along the disk $\mathcal{D}$, we obtain a tangle $t_{L}$ in $S^{3}$. The fundamental rack of a tangle is obtained in a similar way as the fundamental rack of a (classical) link, see \cite[page 210]{EN1}. The fundamental rack of the tangle $t_{L}$ is given by $$R(t_{L})=\left [x_{1},\ldots ,x_{m+d}|\, \textrm{crossing relations of $L$}\right ]\;.$$Observe that by adding identifications on the generators, corresponding to the open arcs of $t_{L}$, we obtain the fundamental rack of the classical link $L'$ in $S^{3}$ that corresponds to $L$: $$R(L')=\left [x_{1},\ldots ,x_{m+d}|\, \textrm{crossing relations of $L$},x_{i}=x_{m+i}\textrm{ for }i=1,\ldots ,d\right ]\;.$$

\begin{corollary}\label {cor2} Let $X$ be a finite rack. \begin{enumerate}
\item If $f\colon S\to X$ defines a coloring of $R(D)$, then $f_{k}\colon \{x_{1},\ldots ,x_{m+d}\}\to X$ is a coloring of $R(t_{L})$ for $k=0,\ldots ,p-1$. 
\item Suppose $L$ is an affine link in $L(p,1)$. If $f_{k}\colon \{x_{1},\ldots ,x_{m+d}\}\to X$ is a coloring of $R(t_{L})$ for $k=0,\ldots ,p-1$, such that the map $f=(f_{0},\ldots ,f_{p-1})$ satisfies the condition $(ii)$ of proposition \ref{prop3}, then $f$ is a coloring of $R(D)$. 
\item Suppose $L$ is a link in $L(p,1)$ which is not affine. If $f_{k}\colon \{x_{1},\ldots ,x_{m+d}\}\to X$ is a coloring of $R(t_{L})$ for $k=0,\ldots ,p-1$, such that the map $f=(f_{0},\ldots ,f_{p-1})$ satisfies the conditions $(ii)$ and $(iii)$ of proposition \ref{prop4}, then $f$ is a coloring of $R(D)$. 
\end{enumerate}
\end{corollary}
\begin{proof}\begin{enumerate}
\item If $f$ defines a coloring of $R(D)$, then the propositions \ref{prop3}(i) and \ref{prop4}(i) imply that $f_{k}$ preserves all the crossing relations of $L$ . It follows that $f_{k}$ defines a coloring of $R(t_{L})$ for $k=0,\ldots ,p-1$.  
\item If $f_{k}\colon \{x_{1},\ldots ,x_{m+d}\}\to X$ is a coloring of $R(t_{L})$ for $k=0,\ldots ,p-1$, then $f_{k}$ preserves all the crossing relations of $L$ . The map $f=(f_{0},\ldots ,f_{p-1})$ therefore satisfies the condition $(i)$ of the proposition \ref{prop3}. If it also satisfies the condition $(ii)$, then it defines a coloring of $R(D)$.  
\item If $f_{k}\colon \{x_{1},\ldots ,x_{m+d}\}\to X$ is a coloring of $R(t_{L})$ for $k=0,\ldots ,p-1$, then $f_{k}$ preserves all the crossing relations of $L$ . The map $f=(f_{0},\ldots ,f_{p-1})$ therefore satisfies the condition $(i)$ of the proposition \ref{prop3}. If it also satisfies the conditions $(ii)$ and $(iii)$, then it defines a coloring of $R(D)$.  
\end{enumerate} 
\end{proof}

Sam Nelson wrote a \verb|Phyton| code which computs all the rack homomorphisms from a fundamental rack of a classical link $L$ to a finite rack $X$. The input of his function \textit{pdhomlist}($PD,M_{X}$) are the planar diagram $PD$ of the link $L$ and the rack matrix $M_{X}$ of the rack $X$. The output are all the colorings of the diagram $PD$ with colors from $X$ which satisfy the crossing relations. Since a representation of the fundamental rack of a (classical) tangle is obtained from its planar diagram in the same way as the representation of the fundamental rack of a classical link, the function also works for tangles. 

In our case, $L$ is a link in the lens space $L(p,1)$ and by the corollary \ref{cor2}, every coloring of $R(D)$ defines a $p$-tuple of colorings of $R(t_{L})$. We may use the function \textit{pdhomlist}($PD,M_{X}$) to find the set of all possible colorings of $R(t_{L})$. Any $p$-tuple from this set will satisfy the condition $(i)$ from the Propositions \ref{prop3} and \ref{prop4}. We must then choose the $p$-tuple $(f_{0},\ldots ,f_{p-1})$ in such a way that it will also satisfy the conditions $(ii)$ and $(iii)$ to define a rack homomorphism from $R(D)\to X$.  
\end{subsection}

\end{section}

\begin{section}{The counting rack invariants of links in $L(p,1)$}\label{sec3}
In this section we recall the counting rack invariants of classical links, defined by Sam Nelson in \cite{SN1}. We show that those invariants may be generalized and applied to the links in the lens spaces $L(p,1)$. 
 
Recall the following result about the diagram equivalence for links in $L(p,q)$:
\begin{Theorem}[\cite{HP},\cite{gma}] \label{th1}Two links $L_{1}$ and $L_{2}$ in $L(p,q)$ with diagrams $D_{1}$ and $D_{2}$ are isotopic if and only if there exists a finite sequence of moves $\Omega _{1}$, $\Omega _{2}$, $\Omega _{3}$ and $SL_{p,q}$ that transform one diagram to the other. 
\end{Theorem}
\begin{figure}[h]
\labellist
\normalsize \hair 3pt
\pinlabel $\Omega _{1}$ at 130 -20
\pinlabel $\Omega _{2}$ at 850 -20
\pinlabel $\Omega _{3}$ at 1770 -20
\endlabellist
\begin{center}
\includegraphics[scale=0.2]{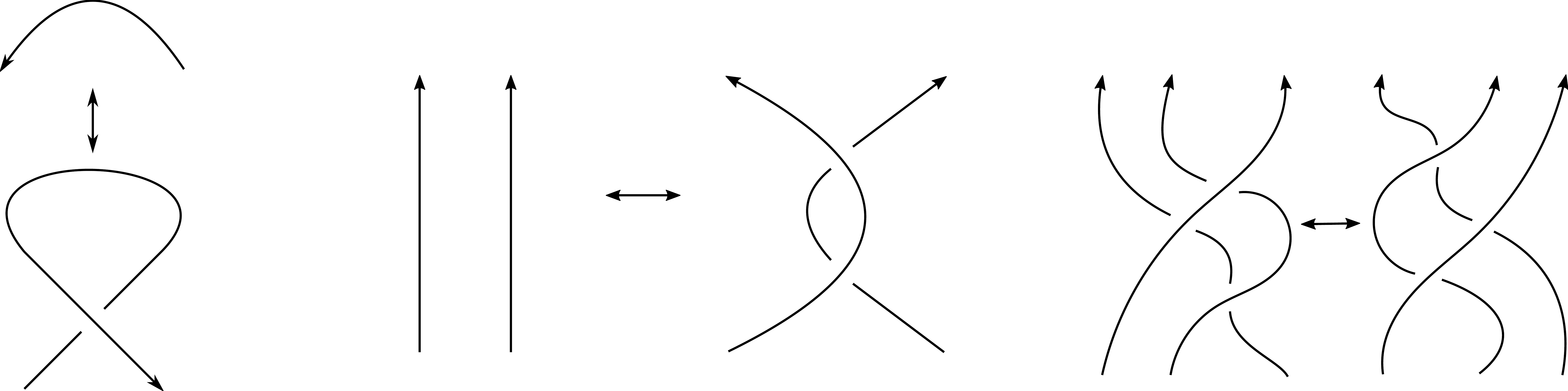}
\caption{The moves $\Omega _{1}$, $\Omega _{2}$ and $\Omega _{3}$}
\label{fig:Reid}
\end{center}
\end{figure}
\begin{figure}[h]
\labellist
\normalsize \hair 3pt
\pinlabel $5/2$ at 30 0
\pinlabel $5/2$ at 680 0
\endlabellist
\begin{center}
\includegraphics[scale=0.4]{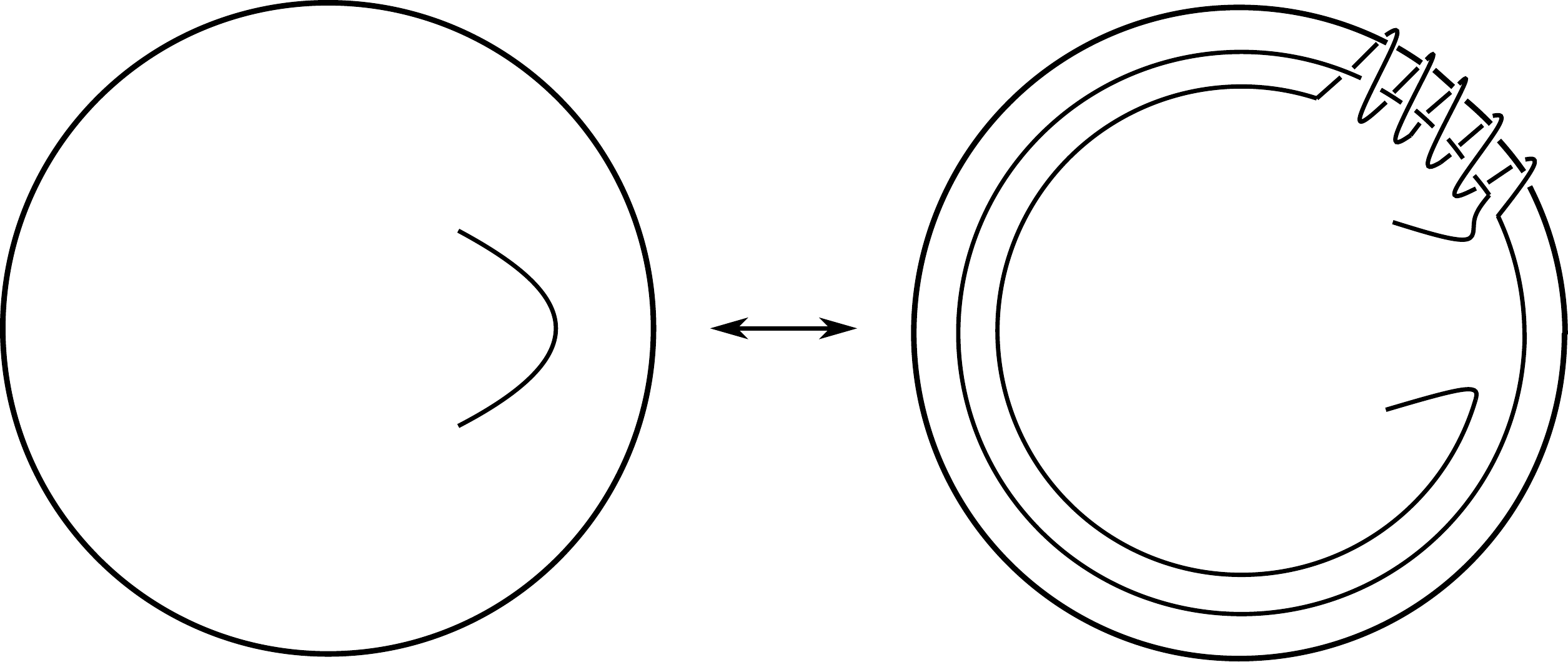}
\caption{The slide move $SL_{5,2}$}
\label{fig:SLpq}
\end{center}
\end{figure}

Let $L$ be an $n$-component  link in $L(p,1)$, and let $X$ be a finite rack with rack rank $N$. Each framing of $L$ is given by a writhe vector $\mathbf{w}\in \ZZ ^{n}$. Denote the blackboard-framed diagram of $L$ with writhe vector $\mathbf{w}\in \mathbb{Z}^{n}$ by $(D,\mathbf{w})$. For each framing $\mathbf{w}$, the diagram $(D,\mathbf{w})$ of the framed link defines its own fundamental rack $R(D,\mathbf{w})$ with its own homomorphism set $Hom(R(D,\mathbf{w}),X)$. Luckily, since $X$ is a finite rack, there are only finitely many framings which produce different homomorphism sets, by the following results of \cite{SN1}:   
\begin{definition}\cite[Definition 4]{SN1}\label{def10} Let $N\in \mathbb{N}$. Two blackboard-framed oriented link diagrams are \textbf{$N$-phone cord equivalent} if one may be obtained from the other by a finite sequence of $\Omega _{2}$ and $\Omega _{3}$ moves and the \textbf{$N$-phone cord move} in the Figure \ref{fig:phone}, where $N$ is the number of loops. 
\end{definition}
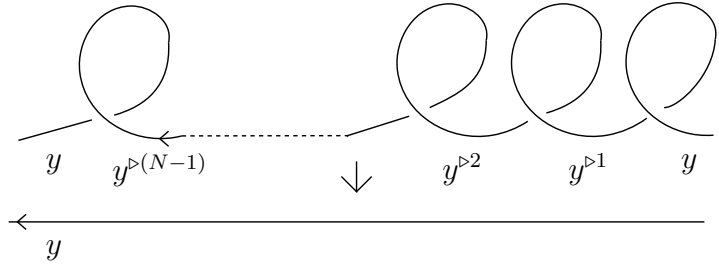
\begin{figure}[h]
\labellist
\normalsize \hair 2pt
\pinlabel $y$ at 1350 130
\pinlabel $y^{\tr 1}$ at 1150 120
\pinlabel $y^{\tr 2}$ at 900 120
\pinlabel $y^{\tr (N-1)}$ at 300 120
\pinlabel $y$ at 90 130
\pinlabel $y$ at 90 -40
\endlabellist
\begin{center}
\includegraphics[scale=0.2]{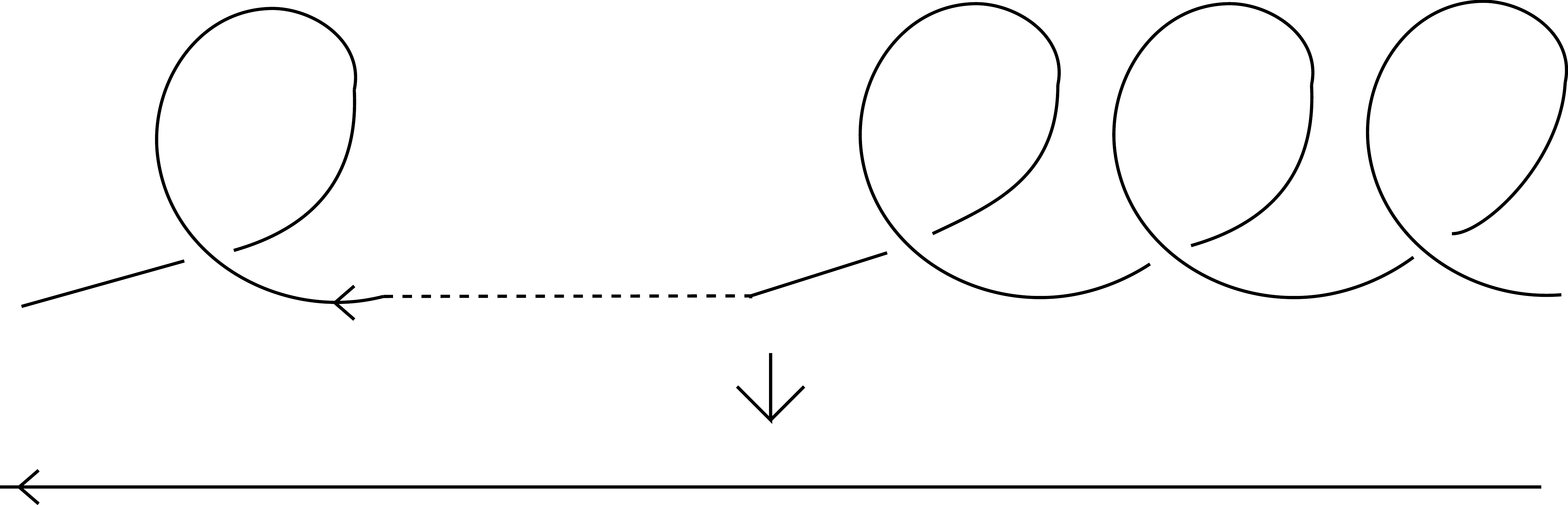}
\caption{$N$-phone cord move}
\label{fig:phone}
\end{center}
\end{figure}
\begin{proposition}\label{prop5} \cite[Proposition 4]{SN1} Let $X$ be a finite rack with rack rank $N$. If two link diagrams $D$ and $D'$ are $N$-phone cord isotopic, then $|Hom(R(D),X)|=|Hom(R(D'),X)|$. 
\end{proposition}
\begin{proof} Suppose the $D'$ is obtained from $D$ by an $N$-phone cord move as shown on the Figure. Every rack homomorphism $f\colon R(D)\to X$ is given by the $p$-tuples $(f_{0}(x),\ldots ,f_{p-1}(x))\in X^{p}$ for every arc $x$ of the diagram $D$. The same coloring also defines a rack homomorphism from $R(D')\to X$: the colors of the additional arcs are given by $f_{i}(y^{\tr k})=f_{i}(y)^{\tr k}$ for $k=1,\ldots ,N-1$ and $f_{i}(y^{\tr N})=f_{i}(y)^{\tr N}=f_{i}(y)$ for $i=0\ldots ,p-1$. 
\end{proof}
It follows from the Proposition \ref{prop5} that two writhe vectors $\mathbf{u},\mathbf{v}\in \mathbb{Z}^{n}$ define the same counting rack invariant if $\mathbf{u}\equiv \mathbf{v}\mod N$ (meaning that $u_{i}\equiv v_{i}\mod N$ for $i=1,\ldots ,n$). The set of writhes of $L$ modulo $N$ can be indexed by the set $(\mathbb{Z}_{N})^{n}$. Let $(D,\mathbf{w})$ denote the blackboard-framed diagram of $L$ with writhe vector $\mathbf{w}\in (\mathbb{Z}_{N})^{n}$.
\begin{definition}\cite[Definition 5]{SN1} \label{def11} Let $L$ be an $n$-component link in $L(p,1)$ and let $X$ be a finite rack with rack rank $N$. The \textbf{integral rack counting invariant} of $L$ with respect to $X$ is $$\Phi _{X}^{\mathbb{Z}}(L)=\sum _{\mathbf{w}\in (\mathbb{Z}_{N})^{n}}|Hom(R(D,\mathbf{w}),X)|\;.$$
\end{definition}
The integer rack counting invariant is an integer, counting the cardinality of the set of rack homomorphisms. By decorating each rack homomorphism by the writhe vector of its associated framing, we may enhance this invariant and replace the integer by a polynomial. For a writhe vector $\mathbf{w}=(w_{1},\ldots ,w_{n})\in (\mathbb{Z}_{N})^{n}$, denote the product $\prod _{k=1}^{n}q_{1}^{w_{1}}q_{2}^{w_{2}}\ldots q_{n}^{w_{n}}$ by the formal variable  $q^{\mathbf{w}}$. Then we may define
\begin{definition}\cite[Definition 6]{SN1} \label{def12} Let $L$ be an $n$-component link in $L(p,1)$ with the writhe vector $\mathbf{w}\in (\mathbb{Z}_{N})^{n}=W$. Let $X$ be a finite rack with rack rank $N$. The \textbf{writhe-enhanced rack counting invariant} of $L$ with respect to $X$ is $$\Phi _{X}^{W}(L)=\sum _{\mathbf{w}\in (\mathbb{Z}_{N})^{n}}|Hom(R(D,\mathbf{w}),X)|q^{\mathbf{w}}\;.$$
\end{definition}
\begin{proposition} \label{prop6} $\Phi _{X}^{\ZZ }$ and $\Phi _{X}^{W}$ are invariants of links in $L(p,1)$ for any finite rack $X$. 
\end{proposition}
\begin{proof} We need to show that 
\begin{xalignat}{1}\label{eq7}\Phi _{X}^{\ZZ }(L_{1})=\Phi _{X}^{\ZZ }(L_{2})\textrm{  and  }\Phi _{X}^{W}(L_{1})=\Phi _{X}^{W}(L_{2})
\end{xalignat} for any two isotopic links $L_{1},L_{2}$ in $L(p,1)$. Let $D_{i}$ be a diagram of $L_{i}$ for $i=1,2$. By the Theorem \ref{th1}, it is enough to check the equalities \eqref{eq7} in the case when $D_{1}$ and $D_{2}$ differ by any one of the moves $\Omega _{1},\Omega _{2},\Omega _{3}$ and $SL_{p,1}$. 

First observe that the functions $\Phi _{X}^{\ZZ }$ and $\Phi _{X}^{W}$ are defined as a sum over all possible framings $\mathbf{w}$ which produce different values of $|Hom(R(D,\mathbf{w}),X)|$ (see proposition \ref{prop5}) and thus are invariant of a given framing of the diagram $D$. We may deduce that if $D_{1}$ and $D_{2}$ differ by an $\Omega _{1}$ move, then the equalities \eqref{eq7} hold.
\begin{figure}[h]
\labellist
\normalsize \hair 2pt
\pinlabel $x$ at -10 60
\pinlabel $y$ at 110 60
\pinlabel $z=y\tr x$ at 430 160
\pinlabel $y$ at 680 300
\pinlabel $x$ at 460 300
\endlabellist
\begin{center}
\includegraphics[scale=0.3]{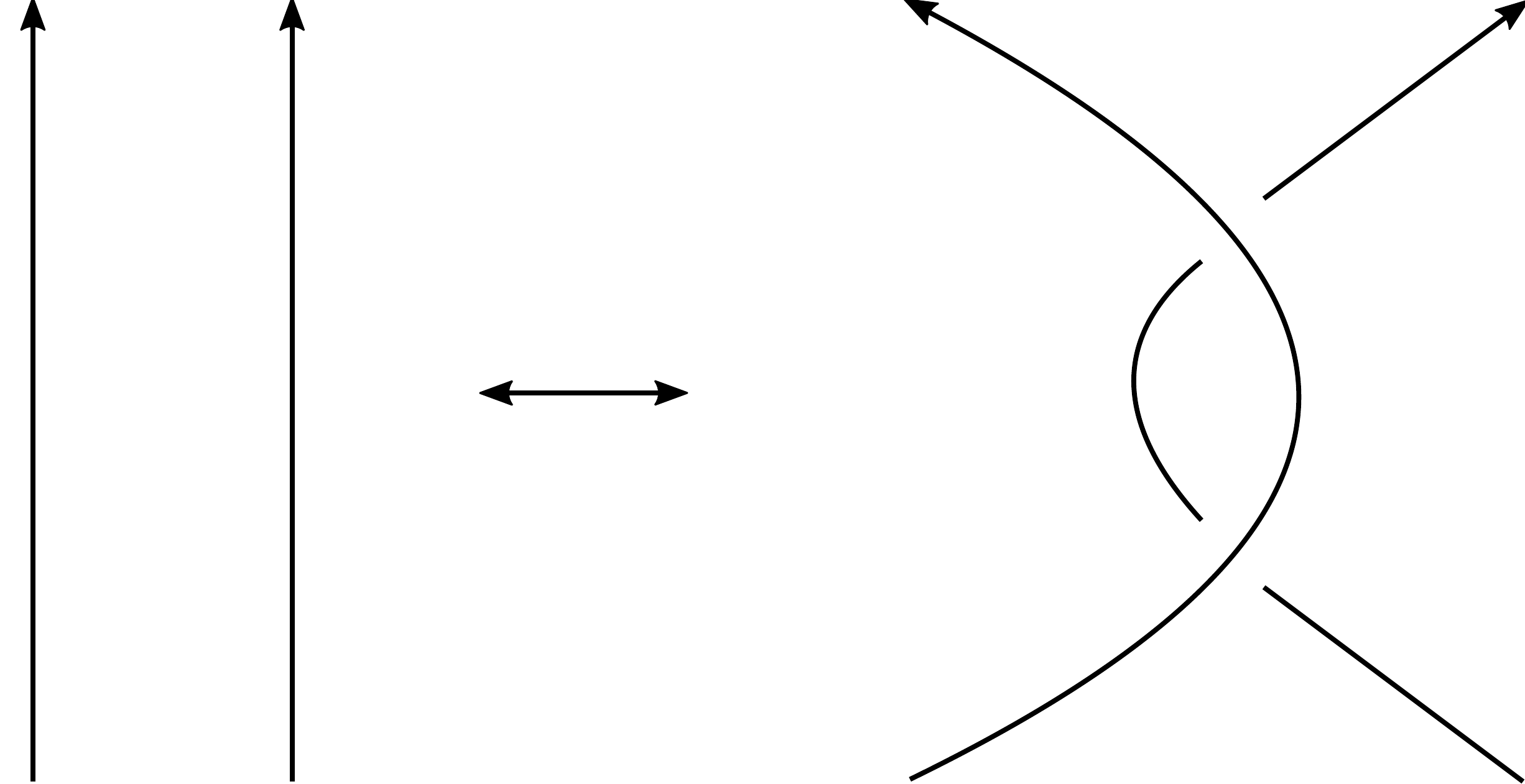}
\end{center}
\end{figure}

Secondly, suppose that the diagram $D_{2}$ is obtained from $D_{1}$ by applying an $\Omega _{2}$ move during which the arc $x$ overcrosses the arc $y$. Then the presentation of $R(D_{1})$ contains the generators $x$ and $y$, and the presentation of $R(D_{2})$ is obtained from the presentation of $R(D_{1})$ by adding a new generator $z$ and a new relation $z=y\tr x$. Since the new generator is written as a product of the previous generators, both presentations present the same rack: $R(D_{1})=R(D_{2})$ and as $\Phi _{X}^{\ZZ }(L_{i})$ and $\Phi _{X}^{W}(L_{i})$ are both functions of the fundamental rack $R(D_{i})$, the equalities \eqref{eq7} follow. 
\begin{figure}[h]
\labellist
\normalsize \hair 2pt
\pinlabel $z$ at -10 20
\pinlabel $y$ at 80 20
\pinlabel $x$ at 190 20
\pinlabel $x\tr y$ at 130 140
\pinlabel $(x\tr y)\tr z$ at 0 420
\pinlabel $y\tr z$ at 140 330
\pinlabel $z$ at 360 20
\pinlabel $y$ at 490 20
\pinlabel $x$ at 590 20
\pinlabel $y\tr z$ at 420 190
\pinlabel $(x\tr z)\tr (y\tr z)$ at 400 420
\pinlabel $x\tr z$ at 500 260
\endlabellist
\begin{center}
\includegraphics[scale=0.3]{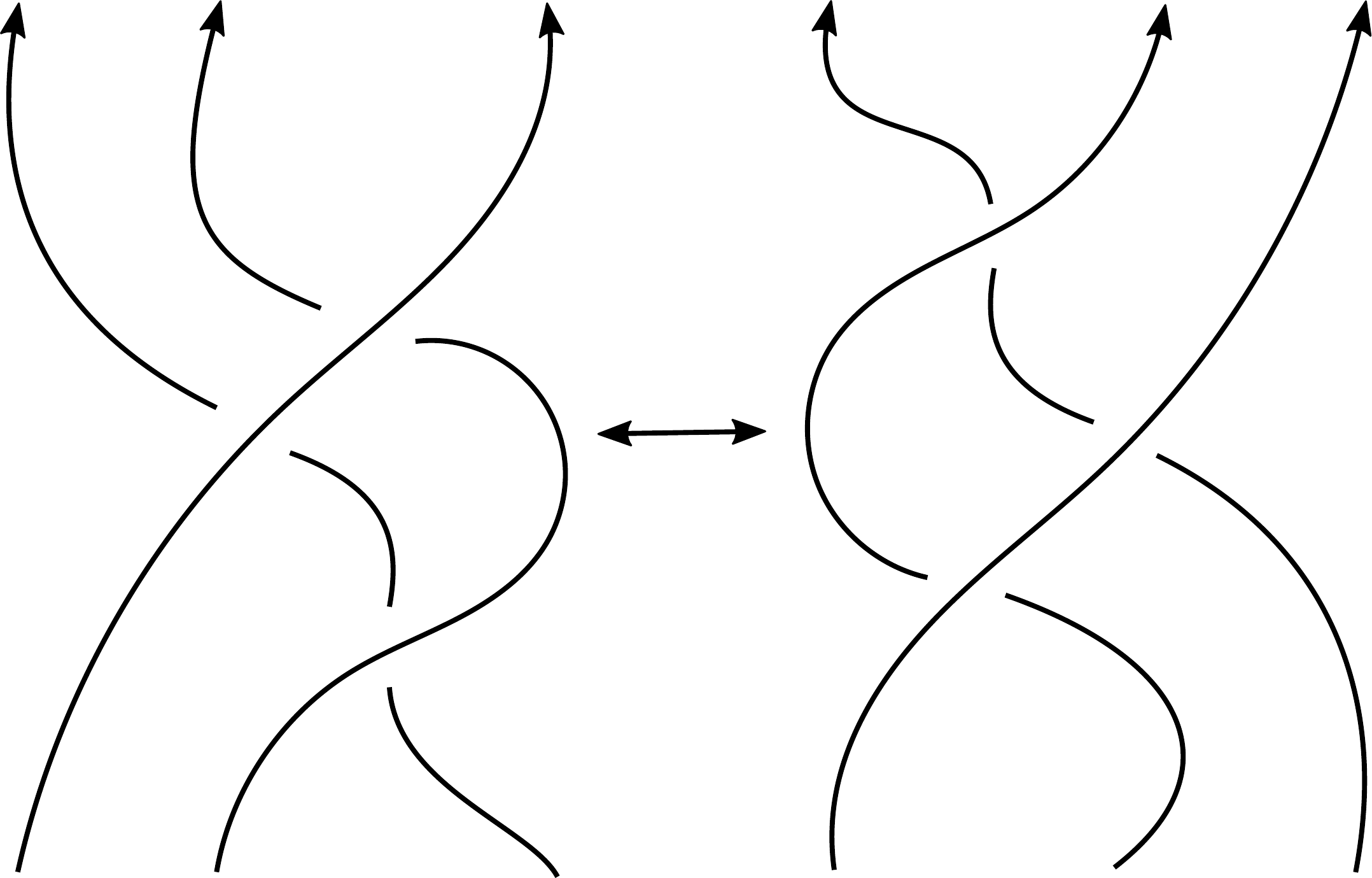}
\end{center}
\end{figure}

Now suppose that $D_{2}$ is obtained from $D_{1}$ by applying an $\Omega _{3}$ move. This means there are three arcs $x,y,z$ representing generators of $R(D_{1})$, whose intercrossings change so that a generator $(x\tr y)\tr z$ in $R(D_{1})$ becomes the generator $(x\tr z)\tr (y\tr z)$ in $R(D_{2})$. By the second rack axiom we have $(x\tr y)\tr z=(x\tr z)\tr (y\tr z)\in R(D_{i})$, therefore $R(D_{1})=R(D_{2})$ and again the equalities \eqref{eq7} follow.   

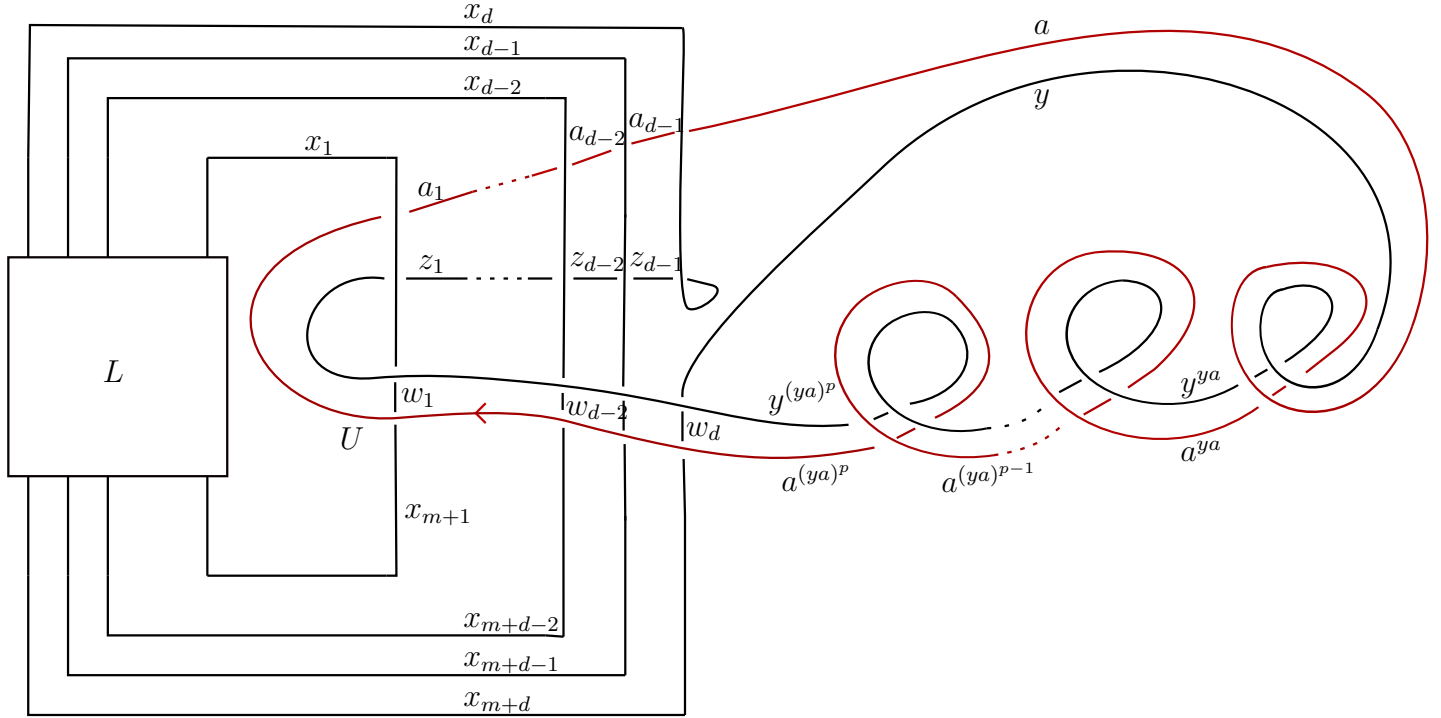
\begin{figure}[h]
\labellist
\normalsize \hair 2pt
\pinlabel $a^{(ya)^{p-1}}$ at 740 180
\pinlabel $a^{ya}$ at 900 200
\pinlabel $y^{ya}$ at 900 250
\pinlabel $a^{(ya)^{p}}$ at 610 180
\pinlabel $y^{(ya)^{p}}$ at 600 240
\pinlabel $a$ at 780 520
\pinlabel $y$ at 780 465
\pinlabel $a_{1}$ [b] at 320 385
\pinlabel $a_{d-2}$ [b] at 445 425
\pinlabel $a_{d-1}$ [b] at 490 435
\pinlabel $z_{1}$ [b] at 320 330
\pinlabel $z_{d-2}$ [b] at 445 330
\pinlabel $z_{d-1}$ [b] at 490 330
\pinlabel $w_{1}$ at 310 240
\pinlabel $w_{d-2}$ at 445 230
\pinlabel $w_{d}$ at 525 215
\pinlabel $x_{d}$ [l] at 340 530
\pinlabel $x_{1}$ [l] at 220 430
\pinlabel $x_{d-2}$ [l] at 340 475
\pinlabel $x_{d-1}$ [l] at 340 505
\pinlabel $x_{m+1}$ [b] at 325 140
\pinlabel $x_{m+d-2}$ [l] at 340 70
\pinlabel $x_{m+d-1}$ [l] at 340 40
\pinlabel $x_{m+d}$ [l] at 340 10
\pinlabel $L$ at 80 260
\pinlabel $U$ at 260 210
\endlabellist
\begin{center}
\includegraphics[scale=0.50]{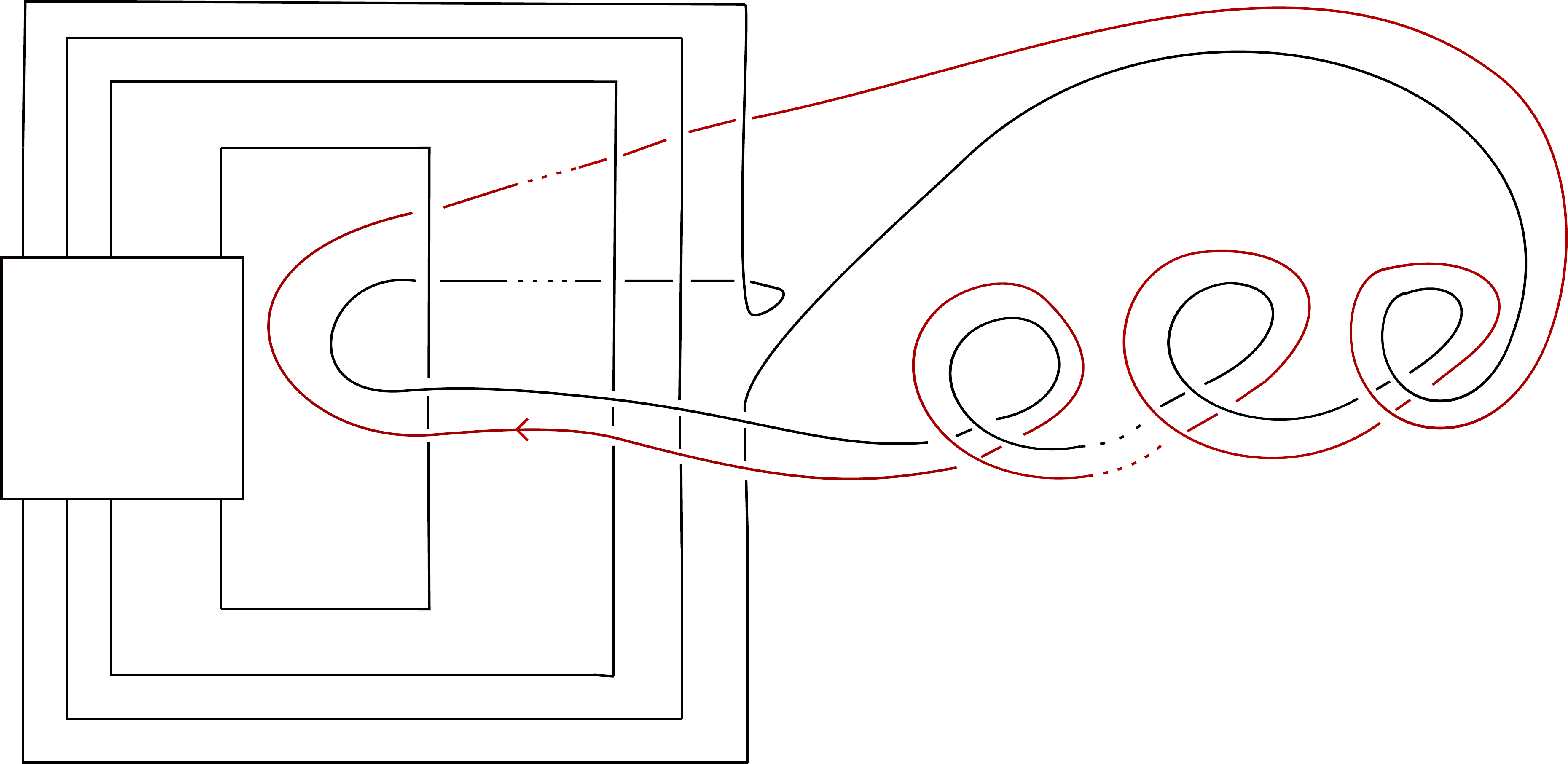}
\caption{The diagram $D_{2}$, obtained by the move $SL_{p,1}$ on a link in $L(p,1)$}
\label{fig:ndiagram}
\end{center}
\end{figure}
Finally, suppose that the diagram $D_{2}$ is obtained from $D_{1}$ by applying an $SL_{p,1}$ move. Let $D_{1}$ be labeled as shown on the Figure \ref{fig:diagram}. Applying the move $SL_{p,1}$, a strand of the link $L$ winds once along the blackboard framing of the surgery curve $U$. In doing so, the presentation 
\begin{xalignat*}{1}
& R(D_{1})= [\{x_{1},\ldots ,x_{m+d}\},\{a\}:\{\textrm{crossing relations of $L$},x_{i}^{a}=x_{m+i}\textrm{ for $i=1,\ldots ,d$}\},\\
& \{ a^{p}\equiv x_{d}^{\epsilon _{d}}x_{d-1}^{\epsilon _{d-1}}\ldots x_{1}^{\epsilon _{1}}\}]
\end{xalignat*} 
of the fundamental rack $R(D_{1})$ changes to a new presentation
\begin{xalignat*}{1}
& R(D_{2})=[\{x_{1},\ldots ,x_{m+d},y,z_{1},\ldots ,z_{d-1},w_{1},\ldots ,w_{d}\},\{a,a_{1},\ldots ,a_{d-1}\}:\{\textrm{crossing relations of $L$},\\ 
& x_{i}^{y^{(ya)^{p}}}=w_{i}\textrm{ for }i=1,\ldots ,d-1,\, y^{y^{(ya)^{p}}}=w_{d},\, w_{i}^{a^{(ya)^{p}}}=x_{m+i}\textrm{ for }i=1,\ldots ,d,\\
& z_{1}^{x_{1}^{\epsilon _{1}}}=y^{(ya)^{p}}, z_{2}^{x_{2}^{\epsilon _{2}}}=z_{1},\ldots ,z_{d-1}^{x_{d-1}^{\epsilon _{d-1}}}=z_{d-2},x_{d}^{x_{d}^{\epsilon _{d}}}=z_{d-1}\},\\
& \{\overline{x}_{1}^{\epsilon _{1}}a_{1}x_{1}^{\epsilon _{1}}\equiv a^{(ya)^{p}},\overline{x}_{2}^{\epsilon _{2}}a_{2}x_{2}^{\epsilon _{2}}\equiv a_{1},\ldots ,\overline{x}_{d-1}^{\epsilon _{d-1}}a_{d-1}x_{d-1}^{\epsilon _{d-1}}\equiv a_{d-2},\overline{x}_{d}^{\epsilon _{d}}ax_{d}^{\epsilon _{d}}\equiv a_{d-1},(ya)^{p}\equiv  x_{d}^{\epsilon _{d}}x_{d-1}^{\epsilon _{d-1}}\ldots x_{1}^{\epsilon _{1}}\}]
\end{xalignat*} obtained from the diagram $D_{2}$ on Figure \ref{fig:ndiagram}. It is easy to see that the new generators $z_{1},\ldots ,z_{d-1}$ and $w_{1},\ldots ,w_{d}$, as well as the operator generators $a_{1},\ldots ,a_{d-1}$ may be omitted from the presentation. The primary relations imply that $x_{d}^{x_{d}^{\epsilon _{d}}x_{d-1}^{\epsilon _{d-1}}\ldots x_{1}^{\epsilon _{1}}}=y^{(ya)^{p}}$ and thus $x_{d}=y^{(ya)^{p}\overline{x}_{1}^{\epsilon _{1}}\ldots \overline{x}_{d}^{\epsilon _{d}}}$, therefore $x_{d}=y$ by the last operator relation. The presentation simplifies to 
\begin{xalignat*}{1}
& R(D_{2})=[\{x_{1},\ldots ,x_{m+d}\},\{a\}:\{\textrm{crossing relations of $L$},x_{i}^{x_{d}a}=x_{m+i}\textrm{ for }i=1,\ldots ,d\},\\
& \{(x_{d}a)^{p}\equiv  x_{d}^{\epsilon _{d}}\ldots x_{1}^{\epsilon _{1}}\}]\;,
\end{xalignat*} which is equivalent to the presentation of $R(D_{1})$. This shows the equalities \eqref{eq7} are valid.   
\end{proof}


\begin{example} \label{ex2} Let $K_{0},K_{1}$ and $K_{2}$ be the knots in $L(3,1)$, given in the Figure \ref{fig:trefoil}, linking the exceptional fibre of $L(3,1)$ zero times, once and twice respectively. Let $R$ be the rack, given by the rack matrix $M_{R}=\begin{bmatrix}
1 & 3 & 2 & 1 & 1 & 1\\
3 & 2 & 1 & 2 & 2 & 2\\
2 & 1 & 3 & 3 & 3 & 3\\
5 & 5 & 5 & 5 & 5 & 5\\
4 & 4 & 4 & 4 & 4 & 4\\
6 & 6 & 6 & 6 & 6 & 6\\
\end{bmatrix}$. Since the rack rank of $R$ equals 2, we need to count the colorings of two different diagrams of $K_{i}$, one with an odd writhe and one with an even writhe. We calculate the counting rack invariants of $K_{i}$ with respect to $R$ as
\begin{xalignat*}{2}
& \Phi _{R}^{\mathbb{Z}}(K_{0})=22 & \Phi _{R}^{W}(K_{0})=12+10q\\
& \Phi _{R}^{\mathbb{Z}}(K_{1})=20 & \Phi _{R}^{W}(K_{1})=10+10q\\
& \Phi _{R}^{\mathbb{Z}}(K_{2})=10 & \Phi _{R}^{W}(K_{2})=6+4q
\end{xalignat*}
\begin{figure}[h]
\labellist
\normalsize \hair 2pt
\pinlabel $K_{0}$ at 640 760
\pinlabel $K_{2}$ at 1200 60
\pinlabel $K_{1}$ at 430 40
\endlabellist
\begin{center}
\includegraphics[scale=0.30]{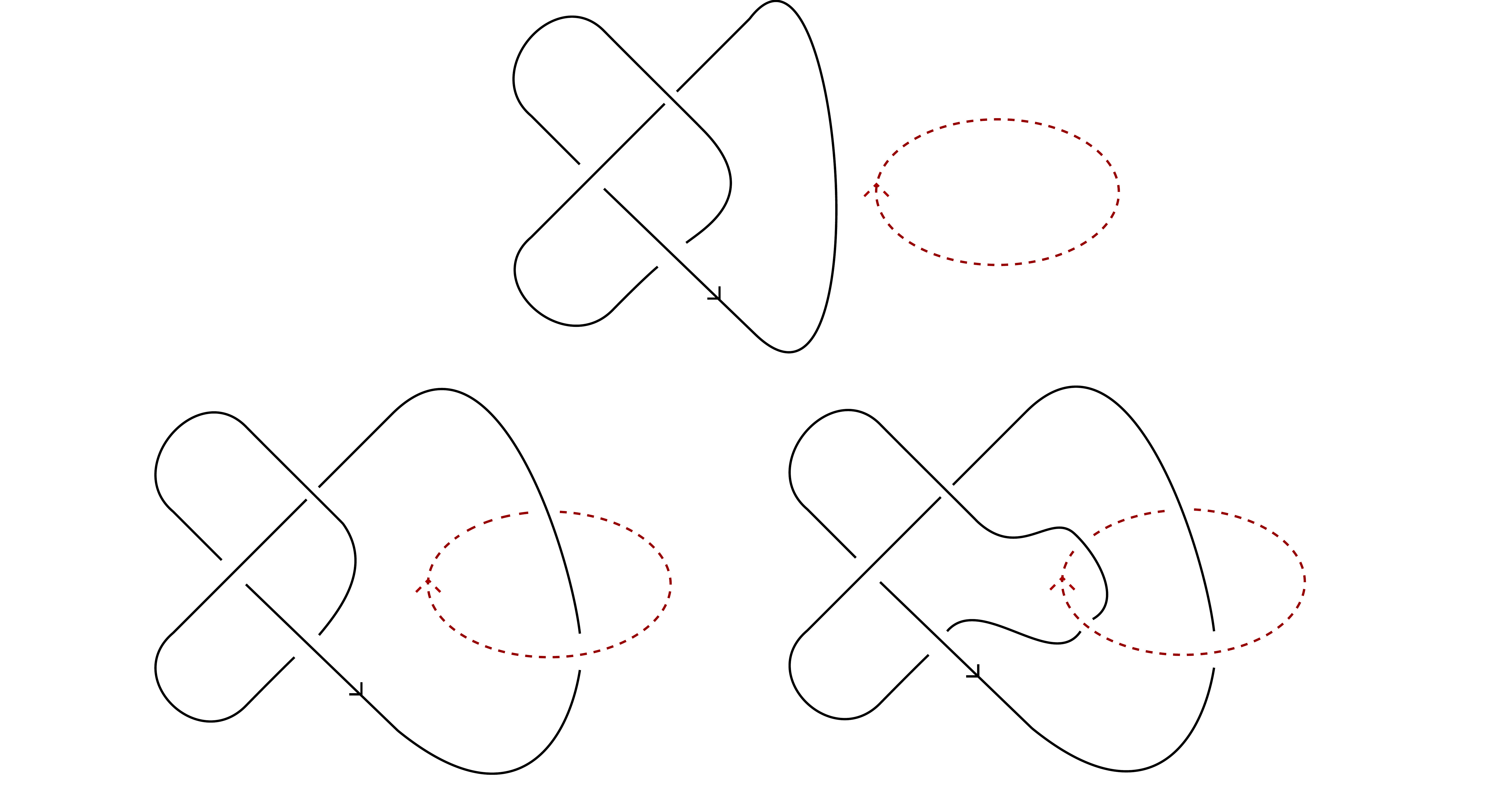}
\caption{The knots $K_{0},K_{1}$ and $K_{2}$ in the example \ref{ex2}}
\label{fig:trefoil}
\end{center}
\end{figure}
\end{example}

\begin{example} \label{ex3} Consider the two links $L_{1},L_{2}$ in $L(3,1)$, given in the Figure \ref{fig:hopf}. 
\begin{figure}[h]
\labellist
\normalsize \hair 2pt
\pinlabel $L_{1}$ at 240 40
\pinlabel $L_{2}$ at 460 40
\endlabellist
\begin{center}
\includegraphics[scale=0.50]{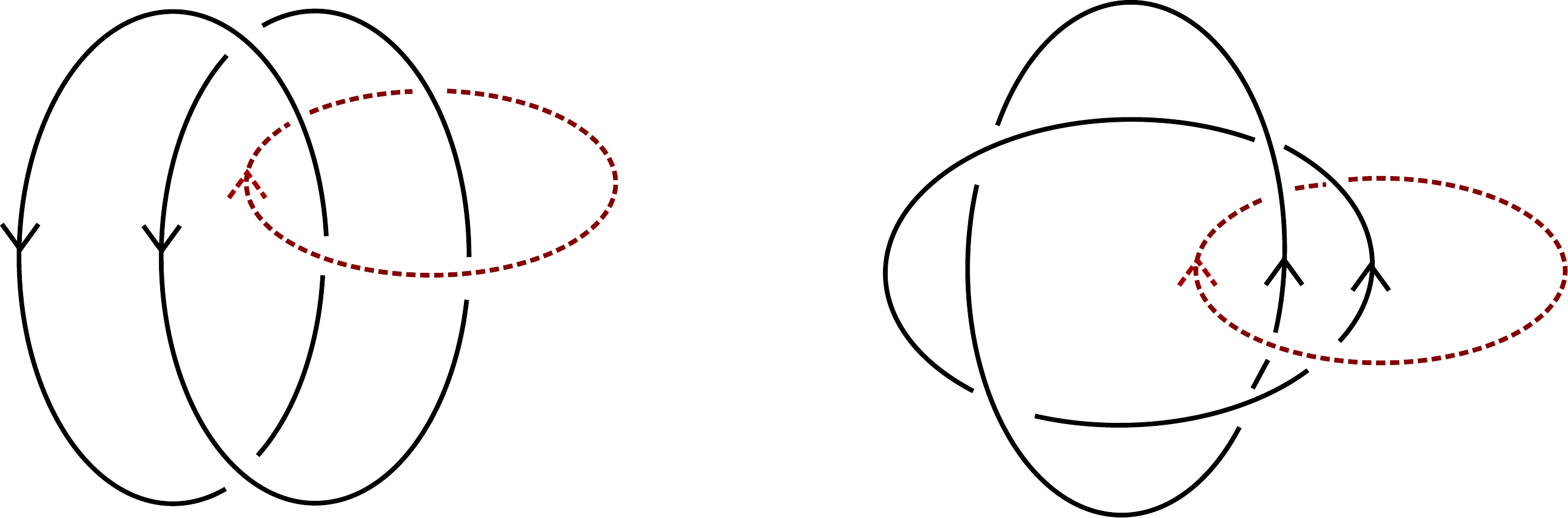}
\caption{The links $L_{1}$ and $L_{2}$ in the example \ref{ex3}}
\label{fig:hopf}
\end{center}
\end{figure}
Let $R$ be the rack of rank 2 with the rack matrix $M_{R}=\begin{bmatrix}
1 & 1 & 1\\
2 & 3 & 3\\
3 & 2 & 2 
\end{bmatrix}$. The counting rack invariants of the links $L_{1}$ and $L_{2}$  with respect to $R$ are calculated as 
\begin{xalignat*}{2}
& \Phi _{R}^{\mathbb{Z}}(L_{1})=12 & \Phi _{R}^{W}(L_{1})=1+3q_{1}+q_{2}+7q_{1}q_{2}\\
& \Phi _{R}^{\mathbb{Z}}(L_{2})=12 & \Phi _{R}^{W}(L_{2})=5+3q_{1}+q_{2}+3q_{1}q_{2}\;,
\end{xalignat*} thus the writhe-enhanced counting invariant distinguishes the links, while the integral rack counting invariant doesn't. 
\end{example}

\end{section}

\newpage

\begin{section}{The rack symmetry invariant}\label{sec4}
In the section \ref{sec3} we described the idea of the counting rack invariants, which starts by visualizing a rack homomorphism of the fundamental rack of a link as a coloring of its diagram. In the case of non-classical links, however, the augmented fundamental rack $R(D)$ carries an additional structure, which arises as a non-trivial action of the fundamental group. 

\begin{lemma} The rack automorphism $A$ of the augmented fundamental rack $R(D)$ corresponds to the generator of the fundamental group of $L(p,1)$. 
\end{lemma}
\begin{proof} Recall the definition of the augmented fundamental rack of a link in the subsection \ref{subs12}. For a framed link $L$ in $Q=L(p,1)$, denote by $N(L)$ the regular neighbourhood of $L$ and let $Q_{0}=\textrm{closure}(L(p,1)-N(L))$. 

By the proposition \ref{prop1}, $R(D)$ is given by the general presentation 
\begin{xalignat*}{1}
& [\{x_{1},\ldots ,x_{m+d}\},\{a\}:\{\textrm{crossing relations of $L$}, x_{i}^{a}=x_{m+i} \textrm{ for $i=1,\ldots ,d$}\},\\
& \{\left (x_{d}^{\epsilon _{d}}\ldots x_{2}^{\epsilon _{2}}x_{1}^{\epsilon _{1}}\right )a\equiv a\left (x_{d}^{\epsilon _{d}}\ldots x_{2}^{\epsilon _{2}}x_{1}^{\epsilon _{1}}\right ),\, a^{p}\equiv x_{d}^{\epsilon _{d}}\ldots x_{2}^{\epsilon _{2}}x_{1}^{\epsilon _{1}}\}] \textrm{ if $d>0$, and }
\end{xalignat*}
\begin{xalignat*}{1} 
& [\{x_{1},\ldots ,x_{m+d}\},\{a\}:\{\textrm{crossing relations of $L$}\},\{a^{p}\equiv 1\}] 
\end{xalignat*} otherwise. By \cite[page 379]{FR}, the presentation of $\pi _{1}(Q_{0})$ may then be obtained as  
\begin{xalignat*}{1}
& \langle x_{1},\ldots ,x_{m+d},a\, :\, \textrm{Wirtinger relations of $L$}, a^{-1}x_{i}a=x_{m+i} \textrm{ for $i=1,\ldots ,d$},\\
& \left (x_{d}^{\epsilon _{d}}\ldots x_{2}^{\epsilon _{2}}x_{1}^{\epsilon _{1}}\right )a=a\left (x_{d}^{\epsilon _{d}}\ldots x_{2}^{\epsilon _{2}}x_{1}^{\epsilon _{1}}\right ),\, a^{p}=x_{d}^{\epsilon _{d}}\ldots x_{2}^{\epsilon _{2}}x_{1}^{\epsilon _{1}}\rangle \textrm{ if $d>0$, and }
\end{xalignat*} $\left <x_{1},\ldots ,x_{m+d},a\, :\, \textrm{Wirtinger relations of $L$}, a^{p}=1 \right >$ otherwise. Note that in the presentation of $R(D)$, the generator $x_{i}$ represents a homotopy class of a path $\gamma $, while in the presentation of $\pi _{1}(Q_{0})$, we denote by $x_{i}$ the homotopy class of the loop $\overline{\gamma }\cdot m_{\gamma (0)}\cdot \gamma $ - this is a loop which encircles the arc $x_{i}$ in the diagram of $L$. 

Consider the following part of the long exact sequence of the pair $(Q,Q_{0})$:
\begin{xalignat*}{1}
& \ldots \longrightarrow \pi _{2}(Q)\longrightarrow \pi _{2}(Q,Q_{0})\overset{\partial }{\longrightarrow }\pi _{1}(Q_{0})\overset{i}{\longrightarrow }\pi _{1}(Q)\longrightarrow \ldots 
\end{xalignat*} Every generator $x_{i}$ of $\pi _{1}(Q_{0})$ represents a boundary of a meridinal disc of the link $L$, thus $x_{i}=\partial d_{i}$ for some $d_{i}\in \pi _{2}(Q,Q_{0})$ and consequently $i(x_{i})=0\in \pi _{1}(Q)$. The operator generator $a$ (which represents the rack automorphism $A$), however, is not trivial in   $Q=L(p,1)$: it represents the boundary of a meridinal disc of the exceptional fiber of $L(p,1)$, thus $i(a)\neq 0$ and $i(a)^{p}=0$.   
\end{proof}

We have shown that the rack automorphism $A$ originates from the action of the fundamental group $\pi _{1}(L(p,1))$ on the augmented fundamental rack $R(D)$. By the definition \ref{def4}, every rack homomorphism $f\colon R(D)\to X$ has an induced action on the image $f(R(D))\subset X$. If $X$ is a finite rack, this action may be given as an element of the symmetric group $S_{|X|}$. 

Let $X$ be a finite rack with the underlying set $\{1,2,\ldots ,m\}$ and let $L$ be a link in $L(p,1)$ with the fundamental rack $R(D)$. Any rack homomorphism $f\colon R(D)\to X$ defines a permutation $\sigma _{f}\in S_{m}$ by 
\[
 \sigma _{f}(k) = 
  \begin{cases} 
   f(A(x)) & \text{if } k=f(x)\;, \\
   k       & \text{if } k\notin f(R(D))\;.
  \end{cases}
\] This permutation is well defined by the definition of the rack homomorphism \ref{def4}. Denote by $\rm{ord}$$(\sigma _{f})$ the order of the permutation $\sigma _{f}\in S_{m}$. Instead of just counting the number of homomorphisms, we might count the number of permutations of a given order that these homomorphisms define. In this way we may further enhance the rack counting invariant of the fundamental rack of a link in $L(p,1)$. 
\begin{definition} Let $X$ be a finite rack with rack rank $N$ and let $L$ be an $n$-component link in $L(p,1)$. The \textbf{rack symmetry invariant} of $L$ with respect to $X$ is $$\Phi _{X}^{Sym}(L)=\sum _{\mathbf{w}\in (\mathbb{Z}_{N})^{n}}\left (\sum _{f\in Hom(R(D,\mathbf{w}),X)}x^{\rm{ord}(\sigma _{f})-1}\right )\;.$$
The \textbf{writhe-enhanced rack symmetry invariant} of $L$ with respect to $X$ is $$\Phi _{X}^{W, Sym}(L)=\sum _{\mathbf{w}\in (\mathbb{Z}_{N})^{n}}\left (\sum _{f\in Hom(R(D,\mathbf{w}),X)}x^{\rm{ord}(\sigma _{f})-1}\right )q^{\mathbf{w}}\;.$$
\end{definition}

\begin{figure}[h]
\labellist
\normalsize \hair 2pt
\pinlabel $K_{0}$ at 450 40
\pinlabel $K_{1}$ at 990 40
\endlabellist
\begin{center}
\includegraphics[scale=0.30]{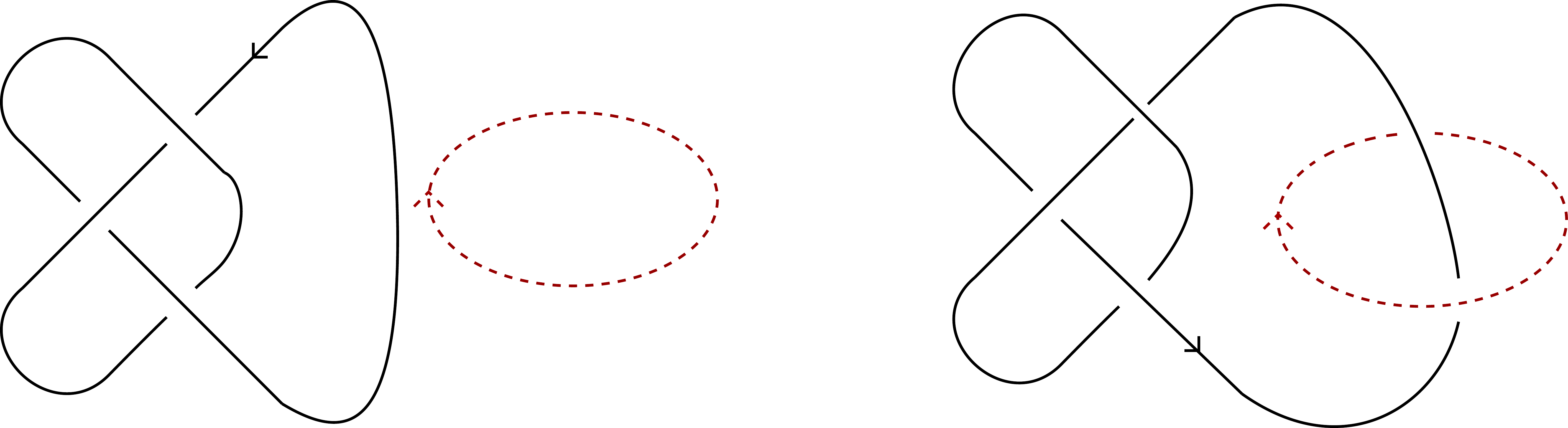}
\caption{Knots $K_{0}$ and $K_{1}$ in the example \ref{ex5}}
\label{fig:trefoil01}
\end{center}
\end{figure}
\begin{example} \label{ex5} Consider the knots $K_{0}$ and $K_{1}$ in $L(2,1)$, whose diagrams are shown in the Figure \ref{fig:trefoil01}. 
Let $R$ be the rack with the rack matrix $M_{R}=\begin{bmatrix}
1 & 4 & 2 & 3\\
3 & 2 & 4 & 1\\
4 & 1 & 3 & 2\\
2 & 3 & 1 & 4
\end{bmatrix}$. The rack rank of $R$ equals 1, which means $R$ is a quandle. The rack counting invariant with respect to $R$ does not distinguish $K_{1}$ from $K_{0}$, since $\phi _{R}^{\ZZ }(K_{0})=\phi _{R}^{\ZZ }(K_{1})=16$. Calculating the rack symmetry invariant, however, gives 
\begin{xalignat*}{2}
& \phi _{R}^{Sym}(K_{0})=16\, , & \phi _{R}^{Sym}(K_{1})=4+12x^{2}\;,
\end{xalignat*}  thus the rack symmetry invariant distinguishes $K_{1}$ from $K_{0}$. 
\end{example}

Another refinement of the rack counting invariant are the rack cocycle invariants. These have been defined for the classical links in  \cite{SN1}, \cite{EN}. The construction might also be generalized by defining the rack cocycle invariant of links in $L(p,1)$.  

The \verb|Phyton| code for computing the rack counting invariants of links in the lens spaces $L(p,1)$ will be available to any interested reader upon request. 

\end{section}


\begin{thebibliography}{1}


\bibitem{SC} J. S. Carter, M. Elhamdadi, M. Grana, M. Saito, {\em Cocycle knot invariants from quandle modules and generalized quandle homology}, Osaka J. Math., \textbf{42}(3) (2005), 499--541.

\bibitem{EN} M. Elhamdadi, S. Nelson, {\em $N$-Degeneracy in rack homology and link invariants}, Hiroshima Math. J., \textbf{42}(1) (2010), 127--142. 

\bibitem{EN1} M. Elhamdadi, S. Nelson, {\em Quandles: an introduction to the algebra of knots}, American Mathematical Society, 2015.

\bibitem{FR} R. Fenn, C. Rourke, {\em Racks and links in codimension two}, J. Knot Theory Ramifications, \textbf{1} (1992), 343--406.

\bibitem{gma} B. Gabrov\v sek and E. Manfredi, {\em On the Seifert  fibered space link group}, Topol. Appl. \textbf{206} (2016), 255--275.

\bibitem{SN2} B. Ho and S. Nelson, {\em Matrices and finite quandles}, Homology, Homotopy and Applications,  \textbf{7(1)} (2005) pp. 197-208.

\bibitem{HP} J. Hoste and J. H. Przytycki, {\em The $(2,\infty )$-skein module of lens spaces; a generalization of the Jones polynomial}, J. Knot Theory Ramifications \textbf{2}(3) (1993), 321--333.

\bibitem{SN1} S. Nelson, {\em Link invariants from finite racks}, Fundamenta Mathematicae, \textbf{225}(1) (2014), 243-258.

\bibitem{MS} M. Saito, C. Smudde et al. {\em Cocycle Invariants of Knots}, http://shell.cas.usf.edu/quandle/Background/qcocyweb.pdf


\end{thebibliography}
\end{document}